\documentclass[a4paper, 10pt]{amsart}
\usepackage{amsthm}
\usepackage{amssymb}
\usepackage{multirow}
\usepackage{booktabs}
\usepackage{array}
\usepackage{graphicx}
\usepackage{graphics}
\usepackage{subfig}
\usepackage{threeparttable}
\usepackage{multicol}
\usepackage{slashbox}
\usepackage{lipsum}
\usepackage{stackrel}
\usepackage{cuted}

\usepackage{ifthen}
\newboolean{ispdf}\setboolean{ispdf}{true}

\ifthenelse{\boolean{ispdf}}
{
 \usepackage[pdftex, colorlinks, allcolors=blue]{hyperref}
}
{
 \newcommand{\url}[1]{{\tt #1}}
 \renewcommand{\includegraphics}[1]{\textbf{Image: #1}}
}

\newcommand{\AAA}{\mathcal{A}}
\newcommand{\BBB}{\mathcal{B}} 

\newcommand{\NNN}{\mathbb{N}}
\newcommand{\RRR}{\mathbb{R}}

\newcommand{\SSS}{S}
\newcommand{\BS}{\BBB_\SSS} 
\newcommand{\eps}{\varepsilon}

\usepackage{geometry}
\newtheorem{theorem}{Theorem}

\newtheorem{lemma}[theorem]{Lemma}
\newtheorem{corollary}[theorem]{Corollary}
\newtheorem{proposition}[theorem]{Proposition}
\theoremstyle{remark}
\newtheorem{remark}[theorem]{Remark}
\newtheorem{example}[theorem]{Example}

\newcommand{\card}{\textnormal{card}}
\newcommand{\Prob}{P}
\newcommand{\Prmu}{\mu}
\newcommand{\diam}{\textnormal{diam}}

\renewcommand\L{\mathrm{L}}
\newcommand\C{\mathrm{C}}

\newcommand\RR{\mathrm{RR}}
\newcommand\DET{\mathrm{DET}}

\newcommand\RATIO{\mathrm{RATIO}}
\newcommand\LAVG{\mathrm{LAVG}}
\newcommand\LMAX{\mathrm{LMAX}}
\newcommand\DIV{\mathrm{DIV}}
\renewcommand\c{\mathrm{c}}

\newcommand\rr{\mathrm{rr}}
\renewcommand\det{\mathrm{det}}

\newcommand\lavg{\mathrm{lavg}}

\newcommand\id{\mathrm{id}}

\newcommand\abs[1]{\left | #1 \right |}

\begin{document}
\bibliographystyle{siam}
\title{Strong laws for recurrence quantification analysis}

\author{M. Grend\'ar}
\author{J. Majerov\'a}
\author{V. \v Spitalsk\'y}

\address{Slovanet a.s., Z\'ahradn\'icka 151, 821 08 Bratislava, Slovakia\newline \indent
Department of Mathematics, Faculty of Natural Sciences, Matej Bel University, Tajovsk\'eho 40, 97401 Bansk\'a Bystrica, Slovakia}

\begin{abstract}
The recurrence rate and determinism are two of the basic complexity measures
studied in the recurrence quantification analysis.
In this paper, the recurrence rate and determinism are expressed in terms of the correlation sum,
and strong laws of large numbers are given for them.
\end{abstract}
\keywords{Strong law of large numbers, recurrence quantification analysis, recurrence rate,
determinism, correlation integral}
\subjclass[2010]{37A50, 60F15 (Primary) 37N99, 60G10 (Secondary)}

\maketitle

\section{Introduction}

The notion of recurrence is one of the fundamental notions in the theory of dynamical
systems. Recurrence plots, introduced by Eckmann, Kamphorst and Ruelle \cite{Eckmann}
in 1987, provide a powerful tool for
recurrence visualization.
The \emph{recurrence plot} of the
trajectory $x_0,x_1,\dots, x_{n-1}$ of a point $x=x_0$
is a black-and-white image
with a pixel $(i,j)$ being black if and only if the trajectory at time $j$ recurs to
the state at time $i$; that is, the points $x_i,x_j$ are close to each other.
The recurrence plot provides a two-dimensional representation of an (arbitrary-dimensional)
dynamical system.

The quantitative study of recurrence plots,
called \emph{recurrence quantification analysis (RQA)},
was initiated by Zbilut and Webber in \cite{Zbilut},
where the authors introduced
several measures of complexity based on the recurrence plot.
Among them, the recurrence rate $\RR$ and the determinism $\DET$
are probably the most important and widely used ones. Their definitions are based on
\emph{diagonal lines} (that is, segments of black points parallel
to the main diagonal), which correspond
to recurrences of parts of the trajectory.

Since the seminal paper \cite{Zbilut},
new RQA tools, quantities and modifications were introduced
and recurrence quantification has been applied in many areas of science,
cf.~\cite{Marwan1} and \cite{Kulkarni}, among others.

Despite its wide use, theoretical properties of recurrence measures
were studied rarely. Asymptotic properties of RQA characteristics
were studied e.g.~in \cite{Faure1,Donges,Faure2,Thiel1,Zou}.
The correlation sum, tightly connected
with the recurrence rate,
as well as derived quantities such as the correlation integral,
correlation dimension and correlation entropy,
were studied extensively, cf.~\cite{Kantz}.
One of the fundamental results, namely
the strong law for correlation sums of ergodic processes, was proved
(by different methods and under different conditions)
in \cite{Pesin1,Pesin2,Aaronson,Serinko,Manning}.
It states that, for a
separable metric space $(Z,d)$ and
a $\mu$-ergodic dynamical system on it,
the correlation sum of the trajectories of almost every point $x\in Z$
with every (up to countably many) $r>0$
converges to the correlation integral
\begin{equation}\label{eq:strong-law-C-intro}
 \lim_{n\to\infty} \C(x,n,r) = \c(r),
 \qquad\text{where}\quad
 \c(r)=\mu\times\mu \{(x,y):\ d(x,y)\le r\} 
\end{equation}
and $\C(x,n,r)=(1/n^2)\cdot\card\{(i,j):\ 0\le i,j<n,\ d(x_i,x_j)\le r\}$.
Recall that the correlation integral $\c(r)$ is just the probability
$P\{d(X,Y)\le r\}$ that two independent random variables $X,Y$ with distribution $\mu$
are $r$-close. It is worth noting that the correlation sum, which measures the level of dependence
in a trajectory, asymptotically turns into the probability
of closeness of two \emph{independent} random variables.

The main purpose of the present paper is to study asymptotic properties
of RQA characteristics for ergodic processes. We start
with a proof of a 
simple formula giving an expression of
the recurrence rate via correlation sums,
see Proposition~\ref{prop:RR-via-C}:
\begin{equation}\label{eq:RR-via-C}
 \RR_{k}^m = k\cdot\C_{k}^m - (k - 1)\cdot\C_{k + 1}^m ,
\end{equation}
where $m$ is the embedding dimension, $k$ is the
prediction horizon 
and $\C^m_k$ and $\RR^m_k$ denote the correlation sum and
recurrence rate, respectively; for the corresponding definitions see
Section~\ref{sec:rqa}.
The relationship (\ref{eq:RR-via-C}) directly permits to express the determinism
$\DET$ in terms of the correlation sums
\begin{equation}\label{eq:DET-via-C}
 \DET_k^m
  = \frac{\RR_{k}^m}{\RR_{1}^m}
  = \frac{k\cdot\C_{k}^m - (k - 1)\cdot\C_{k + 1}^m} {\C_{1}^m}
  \,.
\end{equation}

The bridging formulas (\ref{eq:RR-via-C}) and (\ref{eq:DET-via-C}) enable us to
derive strong laws of large numbers for the recurrence rate and determinism
from that for the correlation sum, see Theorems~\ref{thm:strong-law-RR} and
\ref{thm:strong-law-DET-LAVG-RATIO}. To this end, however, we need to
generalize (\ref{eq:strong-law-C-intro}) to the case when $d$ is a pseudometric
on $Z$, rather than a metric. For pseudometrics induced by Borel maps,
this problem was studied in \cite[Theorem~2]{Serinko}.
In the general case, (\ref{eq:strong-law-C-intro})
was proved by Manning and Simon \cite{Manning};
for details, see Theorem~\ref{thm:C-strong-law} in Section~\ref{sec:strong-law-c}.

We apply the strong laws to iid processes, Markov chains
and autoregressive processes,
and derive explicit formulas for the recurrence integral, asymptotic determinism
and mean diagonal line length
of these processes; see Table~\ref{tab:formulas} and Section~\ref{sec:iid-mc-ar}.
On simulated data we demonstrate the speed of convergence of RQA quantities when
the length $n$ of (the beginning of) the trajectory goes to infinity.

\begin{table}[htp]
\centering
\begin{tabular}{c||cc}
\hline\hline
  & \textnormal{IID}
  & \textnormal{Markov chain}
\\ [0ex] \hline\hline
  \textnormal{recurrence integral}
  & $\alpha^{k + m - 1} \, [k - (k - 1)\,\alpha]$
  & $\alpha\beta^{k + m - 2} \,[k - (k - 1)\,\beta]$
\\
  \textnormal{asymptotic determinism}
  & $\alpha^{k - 1}\, [k - (k - 1)\,\alpha]$
  & $\beta^{k-1}\, [k - (k - 1)\, \beta]$
\\
  \textnormal{mean diagonal line length}
  & $k+{\alpha}/({1 - \alpha}) $
  & $k+\beta/({1 - \beta})$
\\[0ex]\hline\hline
\end{tabular}
\caption{
Formulas for RQA asymptotics for iid processes and Markov chains; here
$\alpha=\c(r)$ and $\beta=\c_2(r)/\c(r)$.}
\label{tab:formulas}
\end{table}

Further,
in Section~\ref{sec:entropy-and-det} we give an example showing that higher entropy of
a process does not necessarily mean smaller (asymptotic)
determinism and that an iid process
can have higher determinism than a non-iid one with the same one-dimensional
marginals. This is a rather unexpected behavior since, in a sense, entropy and
determinism are opposite notions.

We also discuss the problem of choosing the distance threshold $r$. In the literature,
the distance threshold is selected such that, for the embedding dimension $m$,
the recurrence rate attains a fixed level. This rule, however, can lead to
the existence of the so-called spurious structures in recurrence plots of iid processes,
as noted in \cite{Thiel2}, see also \cite[Section~3.2.4]{Marwan1}.
In Section~\ref{sec:spurious} we show why this happens.
For a large embedding dimension $m$ and distance threshold $r_m$
selected by this rule, the determinism $\det^m_k(r_m)$ is close to one even for iid processes.
Hence, the appearance of spurious structures is a direct consequence of the selection rule
fixing the recurrence rate.

The explicit formula for the asymptotic determinism $\det^m_k$ can be stated in terms
of conditional probabilities that $k$ or $(k+1)$ consecutive recurrences occur
given that one recurrence has occurred; see Theorem~\ref{thm:asympt-det-via-cond-prob}.
Hence, if the process under consideration is a Markov one of order $p$,
then over-embedding to dimension $m\ge p$ leaves the asymptotic determinism unchanged;
see Corollary~\ref{cor:Markov-over-embedding}.
This is demonstrated in Section~\ref{sec:ar} for autoregressive processes.
There we discuss possible use of RQA characteristics for estimation of the
order of such processes.

\medskip

The paper is organized as follows.
In Section~\ref{sec:rqa} we recall the definitions of RQA measures
and we prove (\ref{eq:RR-via-C}),
see Proposition~\ref{prop:RR-via-C}.
The strong laws are stated in Section~\ref{sec:strong-law} and,
in Section~\ref{sec:iid-mc-ar}, they are applied
to iid processes, Markov chains
and autoregressive processes.
Relationship between entropy and asymptotic determinism
is discussed in Section~\ref{sec:entropy-and-det}
and the so-called spurious structures
in recurrence plots of iid processes are explained in Section~\ref{sec:spurious}.
In Section~\ref{sec:strong-law-c} we discuss the strong law for correlation sums on pseudometric spaces.

\section{Recurrence quantification analysis (RQA) and correlation sums}\label{sec:rqa}
In this section we recall the definitions of basic RQA measures and of the correlation sum
for (embedded) trajectories of a general $S$-valued process.
To make the notation easier, we write $x_h^m$, $x_h^\infty$ as a shorthand for
$(x_i)_{i=h}^{h+m-1}$, $(x_i)_{i=h}^\infty$, respectively.

Let $S=(S,\varrho)$ be a metric space.
Fix an integer $m\ge 1$
called the \emph{embedding dimension}. Let $S^m$ be
the \emph{embedding space} of $m$-tuples
$s_0^m$ equipped with a metric ${\varrho}^m$
compatible with the product topology.
Natural choices for ${\varrho}^m$ are the Manhattan ($L_1$), Euclidean ($L_2$)
or Chebyshev ($L_\infty$) metrics,
the latter given by
\begin{equation}\label{eq:dist-max}
\varrho^m(s_0^m,t_0^m) = \max_{0\le j<m} \varrho(s_j,t_j),
\end{equation}
but in general we do not restrict $\varrho^m$ to be one of these.

Let $S^\infty$ denote the space
of all sequences $x_0^\infty$ of points from $S$.
This product space
is usually equipped with a metric, say with
$\varrho^\infty(x_0^\infty,y_0^\infty)=\sum_i 2^{-i}\cdot\min\{1,\varrho(x_i,y_i)\}$. In practice, however, we know just (finite) beginnings of trajectories $x_0^\infty$ and thus
we are not able to compute the distance exactly. That is why we use pseudometrics
instead, depending only on the first members of sequences. For an integer
$k\ge 1$, called the \emph{prediction horizon}, a pseudometric
$d^m_k$ on $S^\infty$ is defined by
\begin{equation}\label{eq:dmk-def}
 d^m_k({x}_0^\infty,{y}_0^\infty) = \max_{0\le i<k} \varrho^m({x}_i^m,{y}_i^m).
\end{equation}
For $k=1$ we write simply $d^m$ instead of $d^m_1$.
Notice that $d^m_k$ depends only on the first $(m+k-1)$ members of ${x}_0^\infty,{y}_0^\infty$.

\subsection{RQA measures}\label{sec:preliminaries-rqa}

Fix a sequence $x=x_0^\infty\in S^\infty$ and consider the embedded trajectory
$\tilde{x}=\tilde{x}_0^\infty$, $\tilde{x}_i=x_i^m\in S^m$.
Fix also a distance threshold $r\ge 0$.
For $i,j\in\NNN$ (here $\NNN$ stands for the set of non-negative integers $\{0,1,\dots\}$)
we say that the couple $(i,j)$
is an \emph{$r$-recurrence (in the $m$'th embedding of the trajectory of $x$)} if
$$
 d^m(x_i^\infty, x_j^\infty)=\varrho^m(x_i^m,x_j^m)\le r.
$$
The \emph{recurrence plot} of dimension $n$ is a square $n\times n$ matrix of
zeros and ones, with the entry at $(i,j)$ ($0\le i,j<n$) equal to one if and only if
$(i,j)$ is a recurrence. Usually, the recurrence plot is
visualized by a black-and-white image,
with black pixels representing recurrences.
Let us note that to construct the $n\times n$ recurrence plot (in the $m$'th embedding)
one needs to know only the first $(n+m-1)$ members $x_0^{n+m-1}$ of $x$.

Diagonal lines are basic patterns in the
recurrence plot. We say that $(i,j)$ is a beginning of
a \emph{diagonal line of length $k\ge 1$} in the $n\times n$ recurrence plot
if the following are true:
\begin{itemize}
  \item $0\le i,j\le n-k$;
  \item $(i+h,j+h)$ is a recurrence for every $0\le h<k$;
  \item either at least one of $i,j$ is equal to $0$ or $(i-1,j-1)$ is not a recurrence;
  \item either at least one of $i+k,j+k$ is equal to $n$ or $(i+k,j+k)$ is not a recurrence.
\end{itemize}
For $0<i,j<n-k$ this is equivalent to
\begin{equation*}
 d^m_k(x_i^\infty,x_j^\infty)\le r,
 \quad
 d^m(x_{i-1}^\infty,x_{j-1}^\infty)>r
 \quad\text{and}\quad
 d^m(x_{i+k}^\infty,x_{j+k}^\infty)>r
 .
\end{equation*}
The number of lines of length $k$ in the $n\times n$ recurrence plot
is denoted by $\L^m_k=\L^m_k(x,n,r)$.
Notice that the main diagonal line (i.e.~the case $i=j$) is not excluded,
thus $\L^m_n(x,n,r)=1$;
further, $\L^m_k(x,n,r)=0$ for every $k>n$.

Now fix the {prediction horizon} $k\ge 1$. The
\emph{$k$-recurrence rate $\RR^m_k$} is the percentage of recurrences
contained in diagonal lines of length at least $k$; that is,
\begin{equation}\label{eq:RR-def}
 \RR^m_k = \RR^m_k(x,n,r)
 = \frac{1}{n^2} \sum_{l\ge k} l\cdot \L^m_l .
\end{equation}
The \emph{$k$-determinism $\DET^m_k$}
is the ratio of the $k$-recurrence rate and $1$-recurrence rate
\begin{equation}\label{eq:DET-def}
 \DET^m_k=\DET^m_k(x,n,r)
 =\frac{\RR^m_k}{\RR^m_1}
\end{equation}
(here and throughout we always assume that the denominator is non-zero; otherwise we leave
 the corresponding quantity undefined).
The \emph{$k$-average line length $\LAVG^m_k$}
is the average length of diagonal lines not shorter than $k$
\begin{equation}\label{eq:LAVG-RATIO-def}
 \LAVG^m_k = \frac{\RR^m_k}{(1/n^2) \sum_{l\ge k} \L^m_l} \,;
\end{equation}
again, this characteristic depends also on $x,n,r$.
For the definitions of other RQA characteristics, such as the (Shannon)
 entropy of diagonal line length, trend or measures based on vertical lines, see
 e.g.~\cite{Marwan1}.

\subsection{Correlation sum}\label{sec:preliminaries-C}
Tightly connected with the recurrence rate is the notion of correlation sum,
studied by
Grassberger and Procaccia \cite{Grassberger1,Grassberger2}
in relation to the correlation dimension.
For a sequence $x=x_0^\infty\in S^\infty$,
the embedding dimension $m$, the prediction horizon $k$,
the distance threshold $r\ge 0$ and $n\ge 1$,
the \emph{correlation sum} is defined by
\begin{equation}\label{eq:Cmk-def}
 \C^m_k=\C^m_k(x,n,r) = \frac{1}{n^2} \,\card\{
   (i,j):\ 0\le i,j\le (n-k), \ d^m_k(x_i^\infty,x_j^\infty)\le r
 \} .
\end{equation}
Here, as above, the quantity depends only on the beginning $x_0^{n+m-1}$ of $x$.
$\C^m_k$ measures the relative frequency of
recurrences (in the $m$'th embedding) followed by at least $(k-1)$ other recurrences.
Since, in a diagonal line of length $l\ge k$, just the first $(l-k+1)$ points are followed
by $(k-1)$ other recurrences, it immediately follows that
\begin{equation}\label{eq:C-and-L}
\C_{k}^m
= \frac{1}{n^2} \, \sum_{l\ge k} (l-k+1) \L^m_l
\end{equation}
for every $m,k\ge 1$. Comparison with (\ref{eq:RR-def})  gives the next
statement.

\begin{proposition}\label{prop:RR-via-C}
For $m,k\geq 1$,
\begin{equation*}
\RR_{k}^m = k\cdot\C_{k}^m - (k - 1)\cdot\C_{k + 1}^m.
\end{equation*}
\end{proposition}
\begin{proof}
By (\ref{eq:C-and-L}) and (\ref{eq:RR-def}) we have
\begin{eqnarray*}
n^2 \RR^m_k
&=&
k \left[ n^2\C^m_k - \sum_{l\ge k+1} (l-k+1) \L^m_l \right]
+ \sum_{l\ge k+1} l \L^m_l
\\
&=&
k  n^2\C^m_k - (k-1) \sum_{l\ge k+1} (l-k) \L^m_l
\\
&=&
k  n^2\C^m_k - (k-1) n^2 \C^m_{k+1} ,
\end{eqnarray*}
from which the assertion immediately follows.
\end{proof}

Validity of the previous relation can be also seen from
the following picture
\begin{equation*}
\cdots\circ\underbrace{\bullet\bullet\cdots\bullet}_{a}\underbrace{\bullet\bullet\cdots\bullet}_{b: \abs{b} = k}\circ\cdots
\end{equation*}
of a diagonal line of length $l=\abs{a}+\abs{b}\ge k$.
The $a$-dots are counted in the $k$-recurrence rate as well as in both the $k$
and $(k+1)$-correlation sum.
On the other hand, all of the $b$-dots are counted in the $k$-recurrence rate,
but only the first one is counted in the $k$-correlation sum
and none in the $(k+1)$-correlation sum.
This gives $\RR_{k}^m = k (\C_{k}^m - \C_{k + 1}^m) + \C_{k + 1}^m$, which is equivalent
to the formula from Proposition~\ref{prop:RR-via-C}.

As a corollary of Proposition~1 we can immediately obtain a formula
for the determinism in terms of correlation sums. Since
(\ref{eq:C-and-L}) gives
$(1/n^2)\sum_{l\ge k} \L^m_l=\C^m_{k}-\C^m_{k+1}$, we also
obtain that
\begin{equation}\label{eq:LAVG}
 \LAVG^m_k = k + \frac{\C^m_{k+1}}{\C^m_{k}-\C^m_{k+1}}\,.
\end{equation}

As was noted by many authors,
if the metric $\varrho^m$ in the embedding space is the Chebyshev one (see (\ref{eq:dist-max})),
the embedded recurrence quantities can be expressed in terms of the non-embedded ones.
Let us formulate this as a lemma; there, $\L_k,\C_k,\RR_k$ stand for
$\L^1_k,\C^1_k,\RR^1_k$, respectively

\begin{lemma}\label{lemma:no-embedding}
Let $m,k\ge 1$ and $\varrho^m$ be given by (\ref{eq:dist-max}). Then
\begin{equation*}
 \L^m_k(x,n,r)=\L_{h}(x,n',r),
 \qquad
 \C^m_k(x,n,r)=\left( \frac{n'}{n}   \right)^2 \cdot \C_{h}(x,n',r)
\end{equation*}
and
\begin{equation*}
 \RR^m_k(x,n,r)
 =\left( \frac{n'}{n}   \right)^2 \cdot
  \left[
   \RR_{h}(x,n',r)
   - (m-1) \cdot (\C_{h}(x,n',r)-\C_{h+1}(x,n',r))
  \right]
 ,
\end{equation*}
where $h=k+m-1$ and $n'=n+m-1$.
\end{lemma}
\begin{proof}
By (\ref{eq:dist-max}), $d^m_k(x_i^\infty,x_j^\infty)\le r$
if and only if $\varrho(x_{i+l},x_{j+l})\le r$
for every $0\le l<h$.
Thus the first equality is an immediate consequence of the definition of diagonal lines
and the second one follows from (\ref{eq:C-and-L}).
Further, (\ref{eq:RR-def}) gives
\begin{eqnarray*}
 n^2\RR^m_k(x,n,r)
 &=&
 \sum_{l'\ge h} (l'-(m-1)) \cdot\L_{h}(x,n',r)
\\
 &=&
 (n')^2 \cdot \RR_{h}(x,n',r)
 -(m-1) \cdot \sum_{l'\ge h} L_{h}(x,n',r)
 .
\end{eqnarray*}
Hence, using (\ref{eq:C-and-L}), also the third formula is proved.
\end{proof}

\section{Strong laws for RQA}\label{sec:strong-law}
Here, among other results,
we formulate and prove strong laws of large numbers for the recurrence rate and determinism.
First, the necessary notions and results are summarized.
By a \emph{space} we always mean a topological space.

\subsection{Preliminaries}\label{sec:preliminaries}
Let $Z$ be a space and $\BBB_Z$ be the Borel $\sigma$-algebra on $Z$.
A \emph{(measure-theoretical) dynamical system} is a quadruple $(Z,\BBB_Z,\mu,T)$,
where $\mu$ is a probability measure on $(Z,\BBB_Z)$ and $T:Z\to Z$ is
a (Borel) measurable map which preserves $\mu$, that is, $\mu(T^{-1}(B))=\mu(B)$
for every $B\in \BBB_Z$.
A set $B\in\BBB_Z$ is said to be \emph{$T$-invariant} if $T^{-1}(B)=B$.
We say that $T$ is \emph{$\mu$-ergodic} or that $\mu$ is \emph{$T$-ergodic} if
$\mu(B)\in\{0,1\}$ for every $T$-invariant set $B$.
For $n\in\NNN$, the \emph{$n$-th (forward) iterate $T^n$}
of $T$ is defined recursively by $T^0=\id_Z$
and $T^{n+1}=T\circ T^n$. For $m,n\ge 0$ and $x\in Z$ we write
$T_n^m$ and $T_n^m(x)$ instead of
$(T^i)_{i=n}^{n+m-1}$ and $(T^i(x))_{i=n}^{n+m-1}$, respectively.

Let $\SSS$ be a space with the Borel
$\sigma$-algebra $\BS$.
On the product space $\SSS^\infty$, the Borel $\sigma$-algebra is denoted
by $\BS^\infty$.
An \emph{$\SSS$-valued (discrete time) stochastic process} 
is a sequence $X=X_0^\infty$ of random variables
$X_n:\Omega\to \SSS$ ($n\in\NNN$) defined on a probability space $(\Omega,\BBB,\Prob)$.
The \emph{distribution of the process} $X$
is the measure $\Prmu=\Prmu_X$ on $(\SSS^\infty,\BS^\infty)$ defined by
$\Prmu(F) = \Prob\{X_0^\infty\in F\}$.

The \emph{(left) shift} on $\SSS^\infty$ is the (continuous) map
$T:\SSS^\infty\to \SSS^\infty$ defined by
$$
 T\left(x_0^\infty\right)
 = y_0^\infty,
 \qquad
 \text{where }y_n=x_{n+1} \text{ for every } n\in{\NNN}.
$$
Let $\pi:\SSS^\infty\to \SSS$ denote the projection onto the zeroth coordinate, that is,
$\pi\left(x_0^\infty \right)=x_0$.
If $X$ is a stochastic process with distribution $\mu$, then
the shift $T$ together with the projection $\pi$ and the
measure $\Prmu$ form the \emph{Kolmogorov representation}
of the process $X$. From now on
we always assume that
$X$ is directly given by its Kolmogorov representation, that is,
\begin{equation*}
 (\Omega,\BBB,\Prob)=(\SSS^\infty,\BS^\infty,\mu)
 \quad\text{and}\quad
 X_n=\pi\circ T^n.
\end{equation*}
A process $X$ is \emph{(strictly) stationary}
if its distribution $\Prmu$
is $T$-invariant. The \emph{marginal} of a stationary process $X_0^\infty$ is the distribution
of $X_0$.
A process $X$ is \emph{ergodic}
if every $T$-invariant event has probability either $0$ or $1$.
Thus, a process $X$ is stationary and ergodic if and only if
the dynamical system $(\SSS^\infty,\BS^\infty,\Prmu,T)$ is ergodic.

\subsection{Strong law for correlation sum}\label{sec:correlation-integral}
For a Borel measure $\mu$ on $S^\infty$, $m,k\ge 1$ and $r\ge 0$
define the \emph{correlation integral} $\c^m_k(r)$
by
\begin{equation}\label{eq:cmk-def}
 \c^m_k(r) = \mu\times \mu\{(x,y):\ d^m_k(x,y)\le r\}
  .
\end{equation}
If $\mu$ is the distribution of a process $X_0^\infty$, then
$\c^m_k(r)$
is the probability that, for two independent random
vectors $Y_0^{k+m-1},Z_0^{k+m-1}$ with the distribution equal to that of $X_0^{k+m-1}$,
every $Y_i^m,Z_i^m$ ($i<k$) are $r$-close according to $\varrho^m$:
\begin{equation}\label{eq:cmk-via-prob}
 \c^m_k(r) = \Prmu\left\{
   \varrho^m(Y_i^m,Z_i^m)\le r
   \ \text{for every } 0\le i<k
 \right\}.
\end{equation}

The following theorem, the proof of which is postponed to Section~\ref{sec:strong-law-c},
follows from \cite{Manning}.

\begin{theorem}
\label{thm:Cmk-strong-law}
Let $S$ be a separable metric space,
$X$ be an $S$-valued ergodic stationary process with distribution $\mu$
and $m,k\ge 1$ be integers.
Then for $\mu$-a.e.~trajectory $x\in S^\infty$ of $X$ and for every $r>0$
\begin{equation}\label{eq:Cmk-strong-law}
 \lim_{n\to\infty}\C^m_k(x,n,r) =
 \c^m_k(r) > 0
\end{equation}
provided $\c^m_k$ is continuous at $r$.
\end{theorem}

Notice that $\c^m_k$ is right continuous and non-decreasing,
so it has at most countably many discontinuities; it is
continuous at $r$ if and only if
$\mu\times\mu\{(x,y):\ d^m_k(x,y)=r\}$ is zero.
Further, the convergence in (\ref{eq:cmk-via-prob}) is uniform
over $r$ on any compact interval on which $\c^m_k$ is continuous.

\subsection{Strong laws for RQA}\label{sec:strong-law-rqa}

The purpose of this section is to show that the basic RQA characteristics
converge almost surely to constants, which depend
only on the distribution $\mu$ of the (ergodic) process and on the distance threshold
$r$.
To formulate the results, we introduce the \emph{recurrence
integral $\rr^m_k$}, \emph{asymptotic determinism $\det^m_k$}
and \emph{mean diagonal line length $\lavg^m_k$} for every $r$
by
\begin{equation}\label{eq:RQA-asympt-def}
\begin{split}
 &\rr^m_k(r) = k \cdot\c^m_k(r) - (k - 1) \cdot\c^m_{k+1}(r) ,
\\
 &\det^m_k(r) = \frac{\rr^m_k(r)}{\rr^m_1(r)}
 \qquad\text{and}
\\
 &\lavg^m_k(r) = k + \frac{\c^m_{k+1}(r)}{\c^m_{k}(r) - \c^m_{k+1}(r)}
 \,;
\end{split}
\end{equation}
if $k$ is such that $\c^m_{k}(r) = \c^m_{k+1}(r)>0$ we put $\lavg^m_k(r)=\infty$.
Thus all the quantities are defined for every $r>0$.

Proposition~\ref{prop:RR-via-C} and Theorem~\ref{thm:Cmk-strong-law}
immediately give the following theorem.

\begin{theorem}[Strong law for recurrence rate]\label{thm:strong-law-RR}
Under the assumptions of Theorem~\ref{thm:Cmk-strong-law},
for $\mu$-a.e.~$x\in S^\infty$ and for every (up to countably many) $r>0$,
\begin{equation*}
\lim_{n\to\infty} \RR^m_k(x, n, r) = \rr^m_k(r).
\end{equation*}
\end{theorem}

\begin{theorem}[Strong laws for $\DET$ and $\LAVG$]\label{thm:strong-law-DET-LAVG-RATIO}
Under the assumptions of Theorem~\ref{thm:Cmk-strong-law},
for $\mu$-a.e.~$x\in S^\infty$ and for every (up to countably many) $r>0$,
\begin{equation*}
\lim_{n\to\infty} \DET^m_k(x, n, r) = \det^m_k(r)
\quad\text{and}\quad
\lim_{n\to\infty} \LAVG^m_k(x, n, r) = \lavg^m_k(r)
.
\end{equation*}
\end{theorem}
\begin{proof}
The statements follow since
a.e.-convergence  is preserved by elementary arithmetic operations
provided that, for division, the numerator or denominator is non-zero.
\end{proof}

\begin{remark}\label{rem:RATIO-LMAX}
Theorem~\ref{thm:Cmk-strong-law} can be trivially used to derive strong law also for another
RQA quantity called the \emph{$k$-ratio} defined by
$\RATIO^m_k = \DET^m_k/\RR^m_1$.
Further, for the \emph{maximal diagonal line length $\LMAX^m$} defined by
$$
 \LMAX^m=\LMAX^m(x,n,r) = \max\{l<n:\ \L^m_l(x,n,r)>0\}
$$
using Birkhoff ergodic theorem
one can easily show that, under the assumptions of Theorem~\ref{thm:Cmk-strong-law},
$$
 \lim_{n\to\infty} \LMAX^m(x,n,r) = \infty
$$
for $\mu$-a.e.~$x\in S^\infty$ and for every $r>0$.
As a corollary we immediately have that the reciprocal value
$\DIV^m=1/\LMAX^m$ called the \emph{divergence}
converges almost surely to zero.
\end{remark}

\begin{remark}
Recurrence measures as well as correlation sums are often defined using
strict inequalities $d^m_k(x,y)< r$, and/or with excluding the main diagonal $i=j$.
Clearly, the latter has no effect on
asymptotic properties, that is,
Theorems~\ref{thm:Cmk-strong-law}--\ref{thm:strong-law-DET-LAVG-RATIO}
remain true also in this case. When one uses strict inequalities, then
again the results are valid provided strict inequality is used also in
the definition (\ref{eq:cmk-def}) of the correlation integral. The relationship
between this new ``open'' correlation integral and the used ``closed'' one is straightforward,
see \cite[Remark~2.2]{Pesin2}.
\end{remark}

\begin{remark}
As can be seen from Theorem~\ref{thm:C-strong-law}, Theorem~\ref{thm:Cmk-strong-law}
is valid with $d^m_k$ replaced by any separable Borel
pseudometric $d$ on $S^\infty$. For example, $d$ can be defined via order patterns (cf.~\cite{Amigo}):
$d(x_0^\infty,y_0^\infty) = 1$
if $x_0^m,y_0^m$ have the same order pattern, $d(x_0^\infty,y_0^\infty) = 0$ otherwise.
In this way we obtain strong laws for RQA characteristics based on order patterns recurrence plots.
\end{remark}

As for ``empirical'' RQA quantities (see Lemma~\ref{lemma:no-embedding}), the dependence
of asymptotic ones on the
embedding dimension $m$ is straightforward provided the maximum metric is used.

\begin{lemma}\label{lemma:no-embedding-asympt}
Let $m,k\ge 1$, $r\ge 0$ and $\varrho^m$ be given by (\ref{eq:dist-max}).
Then
\begin{equation*}
 \c^m_k=\c_{h},
 \quad
 \rr^m_k = \rr_{h} - (m-1)(\c_{h}-\c_{h+1})
 \quad\text{and}\quad
 \lavg^m_k = \lavg_{h}-(m-1),
\end{equation*}
where $h=k+m-1$.
\end{lemma}
\begin{proof}
The first equality follows from (\ref{eq:cmk-def}) and the definition
(\ref{eq:dmk-def}) of $d^m_k$.
The others are then consequences of (\ref{eq:RQA-asympt-def}).
\end{proof}

\subsection{Asymptotic determinism via conditional probabilities}
Here we assume (\ref{eq:dist-max}).
For $h,l\ge 1$ and $r>0$ define the \emph{conditional correlation integral} by
\begin{equation*}
 \c_{l|h} (r)
 = \frac{\c_{h+l}(r)}{\c_{h}(r)}
 = \frac
   {\mu\times\mu\{(y,z):\ d_{h+l}(y,z)\le r \}}
   {\mu\times\mu\{(y,z):\ d_{h}(y,z)\le r \}} \,.
\end{equation*}
Particularly, if $\mu$ is the distribution of an ergodic stationary process $X_0^\infty$
 and
$Y_0^{h+l},Z_0^{h+l}$ are independent random vectors with the distribution equal to that of
$X_0^{h+l}$, then $\c_{l|h} (r)$ is the conditional probability
\begin{equation*}
\c_{l|h} (r)
= \Prmu\left\{
 \varrho^l(Y_h^l,Z_h^l)\le r 
 ~|~
 \varrho^h(Y_0^h,Z_0^h)\le r 
\right\}.
\end{equation*}
Thus, $\c_{l|h} (r)$ is the probability that $h$ consecutive recurrences
are followed by at least $l$ other ones.
In view of this we have the following interesting
expression of asymptotic determinism in terms of conditional probabilities.

\begin{theorem}\label{thm:asympt-det-via-cond-prob}
Under (\ref{eq:dist-max}), the asymptotic determinism
can be expressed via a linear combination of conditional correlation integrals
\begin{equation*}
\det^m_k = k\cdot \c_{k-1|m} - (k-1) \cdot \c_{k|m} .
\end{equation*}
\end{theorem}

Consider now the special case of (ergodic stationary)
Markov processes of order $p\ge 1$. Then for every $m\ge p$ one has
$\c_{l|m}=\c_{l|p}$. That is, over-embedding has no effect on the asymptotic determinism.

\begin{corollary}\label{cor:Markov-over-embedding}
For every (ergodic stationary) Markov process of order $p\ge 1$ and for every
$m\ge p$,
$\det^m_k=\det^p_k$.
\end{corollary}

\section{Asymptotic RQA measures for some processes}\label{sec:iid-mc-ar}
Now we present some applications of the asymptotic results obtained in the previous section.
We assume that $S=(S,\varrho)$ is a
separable metric space and $S^\infty$ is
equipped with the pseudometric $d^m_k$ given by (\ref{eq:dmk-def}), where
$m$ is the embedding dimension, $k$ is the prediction horizon and
the embedding metric $\varrho^m$ is given by (\ref{eq:dist-max}).
We also assume that $X_0^\infty$ is
(a Kolmogorov representation of) an ergodic stationary $S$-valued process.
In the following we derive explicit formulas for asymptotic
RQA measures for some classes of processes.
To make the paper self-contained we include here also the proofs, though
the results (at least for correlation integrals) are known.
The convergence is demonstrated by simulation studies.
We start with the simplest case of iid processes.

\subsection{IID processes}\label{sec:iid}

\begin{proposition}\label{prop:det-iid}
Let $X_0^\infty$ be an iid process. Then, for $m,k\ge 1$ and $r \ge 0$,
\begin{equation*}
 \c^m_k(r) = \alpha^{m+k-1} ,
 \qquad\text{where } \alpha=\c(r)
 .
\end{equation*}
Hence
\begin{equation*}
\det^m_k(r) = \alpha^{k-1}\,[k - (k - 1)\,\alpha]
\qquad\text{and}\qquad
\lavg^m_k(r) = k+\frac{\alpha}{1-\alpha}
\end{equation*}
do not depend on the embedding dimension $m$.
\end{proposition}

\begin{proof}
By Lemma~\ref{lemma:no-embedding-asympt} we may assume that $m=1$.
Let $Y_0^k,Z_0^k$ be independent random vectors with the distribution equal to that of $X_0^k$.
Then, for every $r>0$,
\begin{eqnarray*}
\c_k(r)
&=& \Prmu\{\varrho_k(Y_0^k,Z_0^k)\le r\}
= \Prmu\{\varrho(Y_i,Z_i)\le r \text{ for every } 0\le i<k\}
\\
&=& \prod_{0\le i<k}\Prmu\{\varrho(Y_i,Z_i)\le r\}
= [\c(r)]^k .
\end{eqnarray*}
Thus the first statement is proved. The rest follows from
the definitions (\ref{eq:RQA-asympt-def}) of $\det^m_k$ and $\lavg^m_k$.
\end{proof}

\begin{figure}[ht]
  \centering
  \includegraphics[width=4.5cm]{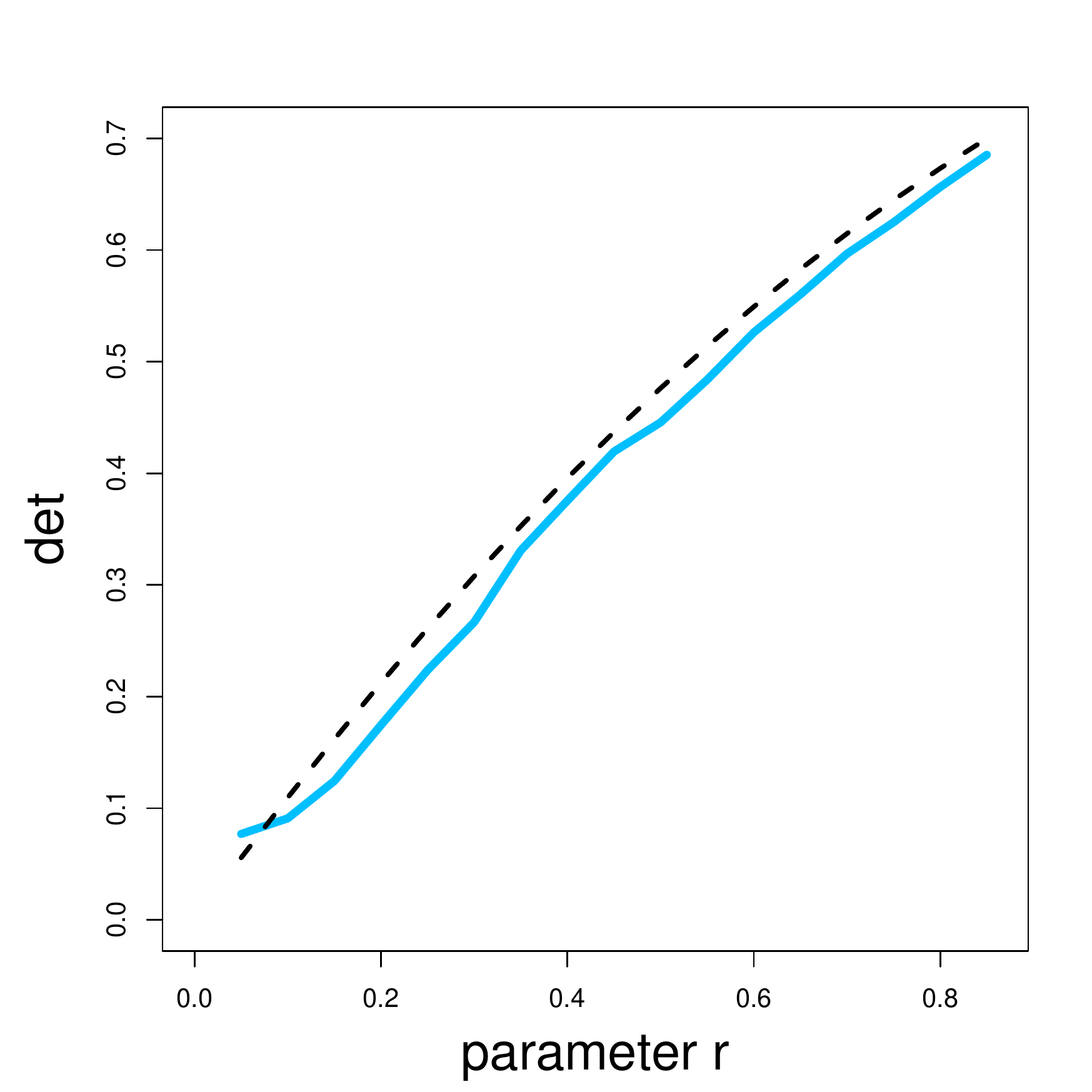}
  \includegraphics[width=4.5cm]{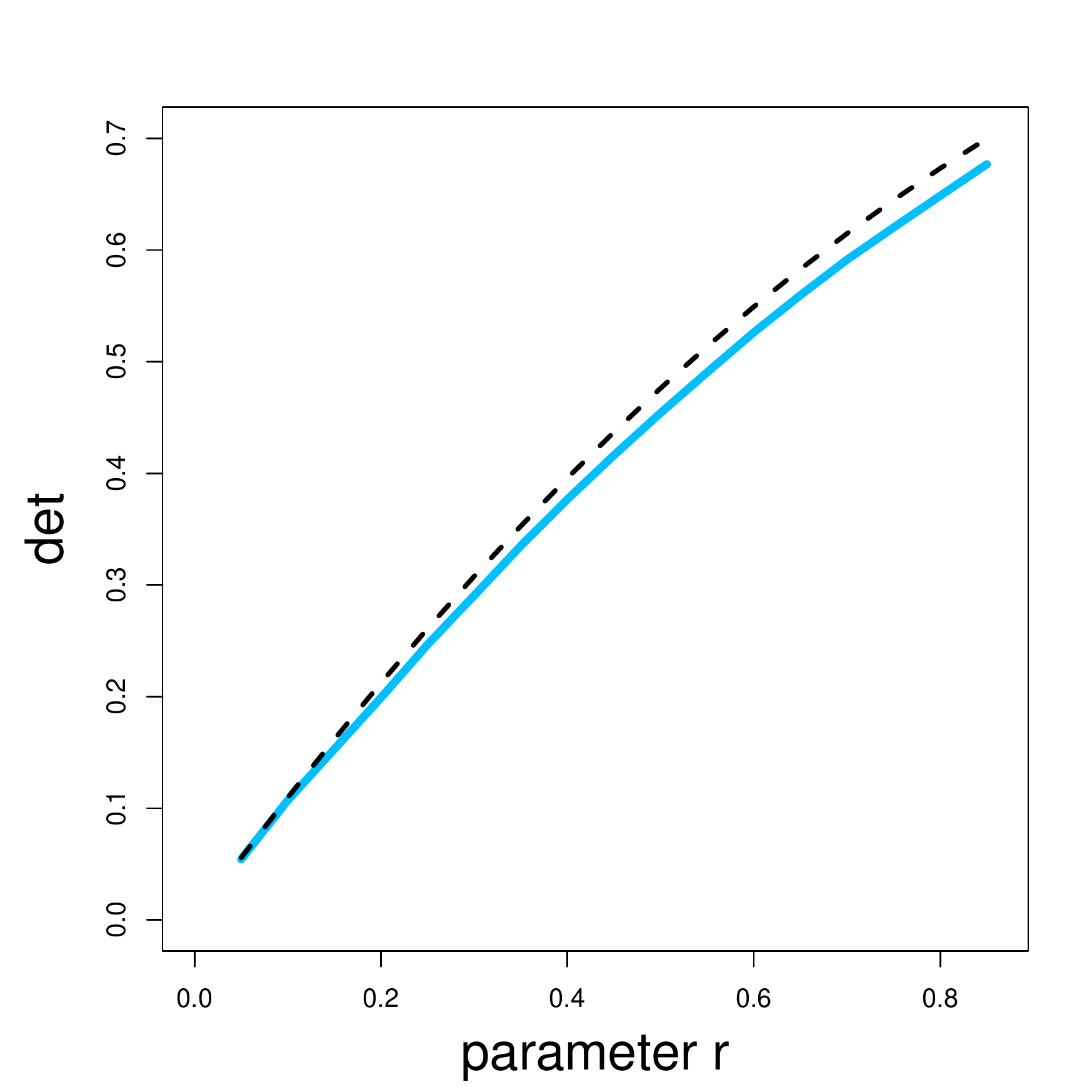}
  \includegraphics[width=4.5cm]{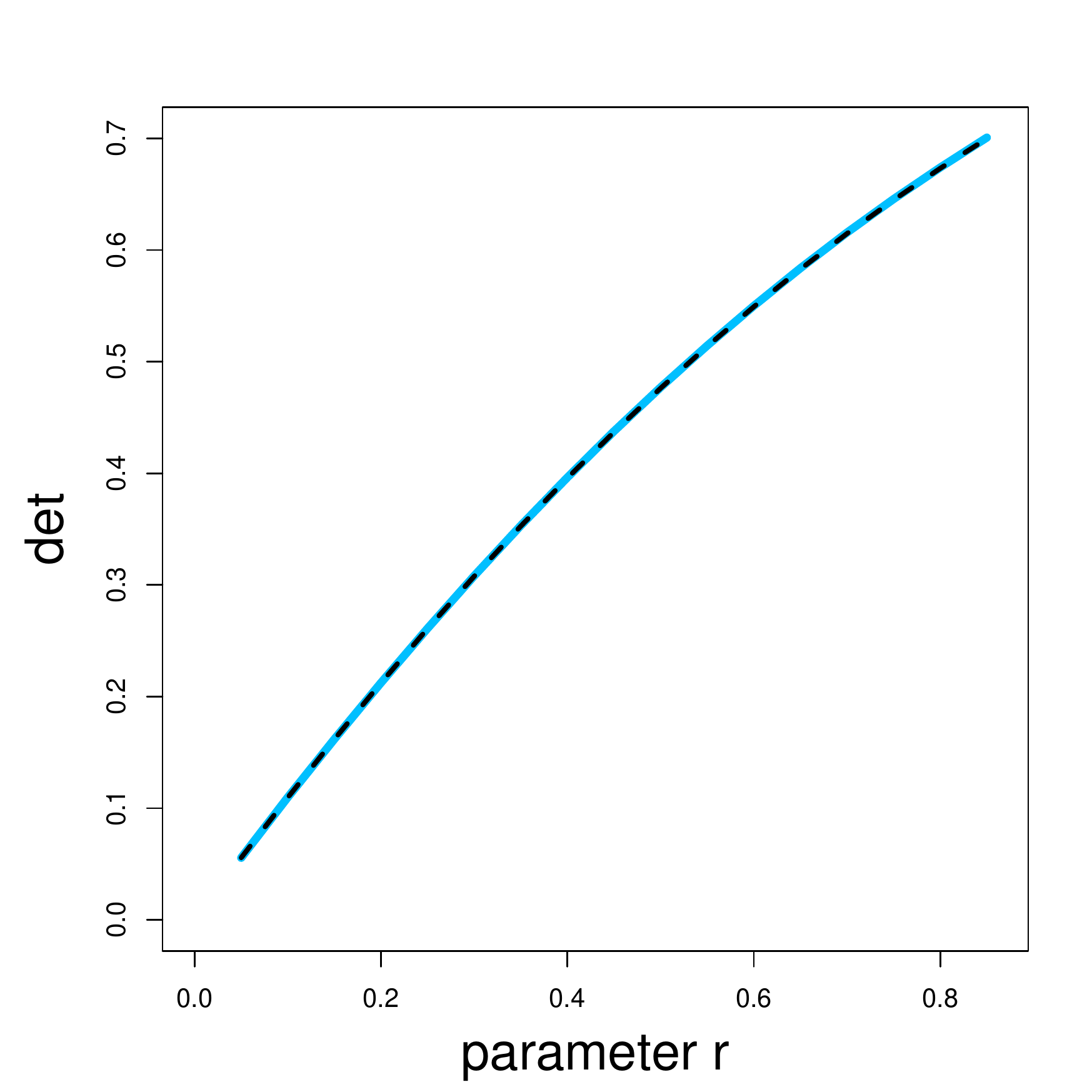}\\
  \caption{Convergence of the empirical determinism (solid line)
    to the asymptotic determinism (dashed line)
    for an iid process with distribution $N(0,1)$;
    $n=100, 1\,000, 10\,000$.
  }
  \label{fig:convergence-iid}
\end{figure}

For example, the asymptotic
determinism of a Gaussian iid process with variance $\sigma^2$  is
$$
 \det^m_k(r) = 2k
 \left[ 2\Phi\left(r'\right) - 1 \right]^{k-1}
 \cdot \left[1-\tfrac{1}{2k} - \Phi\left(r'\right) \right] ,
 \qquad\text{where}\quad
 r'=\frac{r}{\sqrt 2\sigma}
$$
and $\Phi$ is the distribution function of the standard normal distribution.
To see this, use that for iid Gaussian random variables $Y,Z$ with variance $\sigma^2$,
$Y-Z\sim N(0,2\sigma^2)$ and so
$\c(r)=\Prmu\{\abs{Y-Z}\le r\} = 2\Phi\left(r'\right) - 1$.

Figure~\ref{fig:convergence-iid} illustrates the convergence of the empirical determinism
(with $m=1$ and $k=2$) to the asymptotic one for a Gaussian iid process.

\subsection{Markov chains}\label{sec:markov}
Let $S=\{0,1,\dots,q-1\}$ be a finite space equipped with the discrete metric $\varrho$
(that is, $\varrho(x,y)=1$ if $x\ne y$ and $\varrho(x,y)=0$ for $x=y$).
Consider an $S$-valued Markov chain $X_0^\infty$ with the transition matrix
$P=(p_{st})_{s,t=0}^{q-1}$ and the stationary distribution $\pi=(\pi_0,\dots,\pi_{q-1})'$.
Recall that $\pi'P=\pi'$ and
that $X_0^\infty$ is ergodic if and only if the matrix $P$ is transitive or, equivalently,
the probability of the transition from any state $s$ to any state $t$ in a finite time is non-zero.
The formulas for asymptotic values of RQA characteristics of Markov chains
are given in the following proposition. As in the iid case, also here we
can see that both the determinism and mean diagonal line length
do not depend on the embedding dimension.
(Notice that, in this discrete setting, only the
distance threshold $r$ less than $1$ needs to be considered and,
for $0\le r<1$, RQA quantities do not depend
on $r$.)

\begin{proposition}\label{prop:det-markov}
Let $X_0^\infty$ be a finite-valued Markov chain
with the transition matrix $P$ and the stationary distribution $\pi$.
Then, for $r\in[0,1)$,
\begin{equation*}
\c^m_k(r) = \alpha \beta^{k+m-2} ,\qquad
\det^m_k(r)= \beta^{k-1}\, [k - (k-1)\,\beta]
\qquad\text{and}\qquad
\lavg^m_k(r)=k+\frac{\beta}{1-\beta}
\end{equation*}
where $\alpha=\pi'\pi$  and $\beta = (\pi'\textnormal{diag}(PP')\pi) / \alpha$.
\end{proposition}

\begin{figure}[ht!!!]
  \centering
  \includegraphics[width=4.5cm]{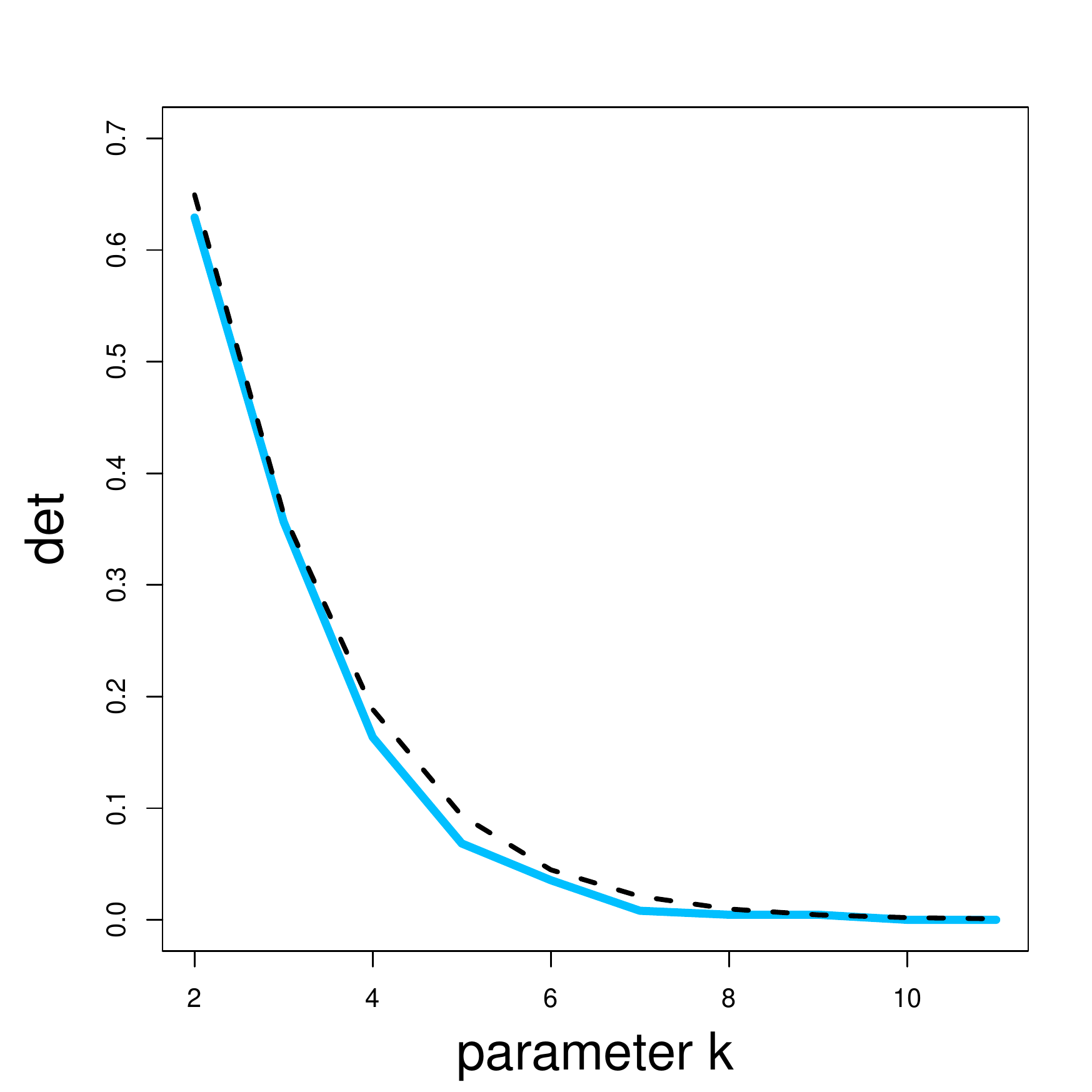}
  \includegraphics[width=4.5cm]{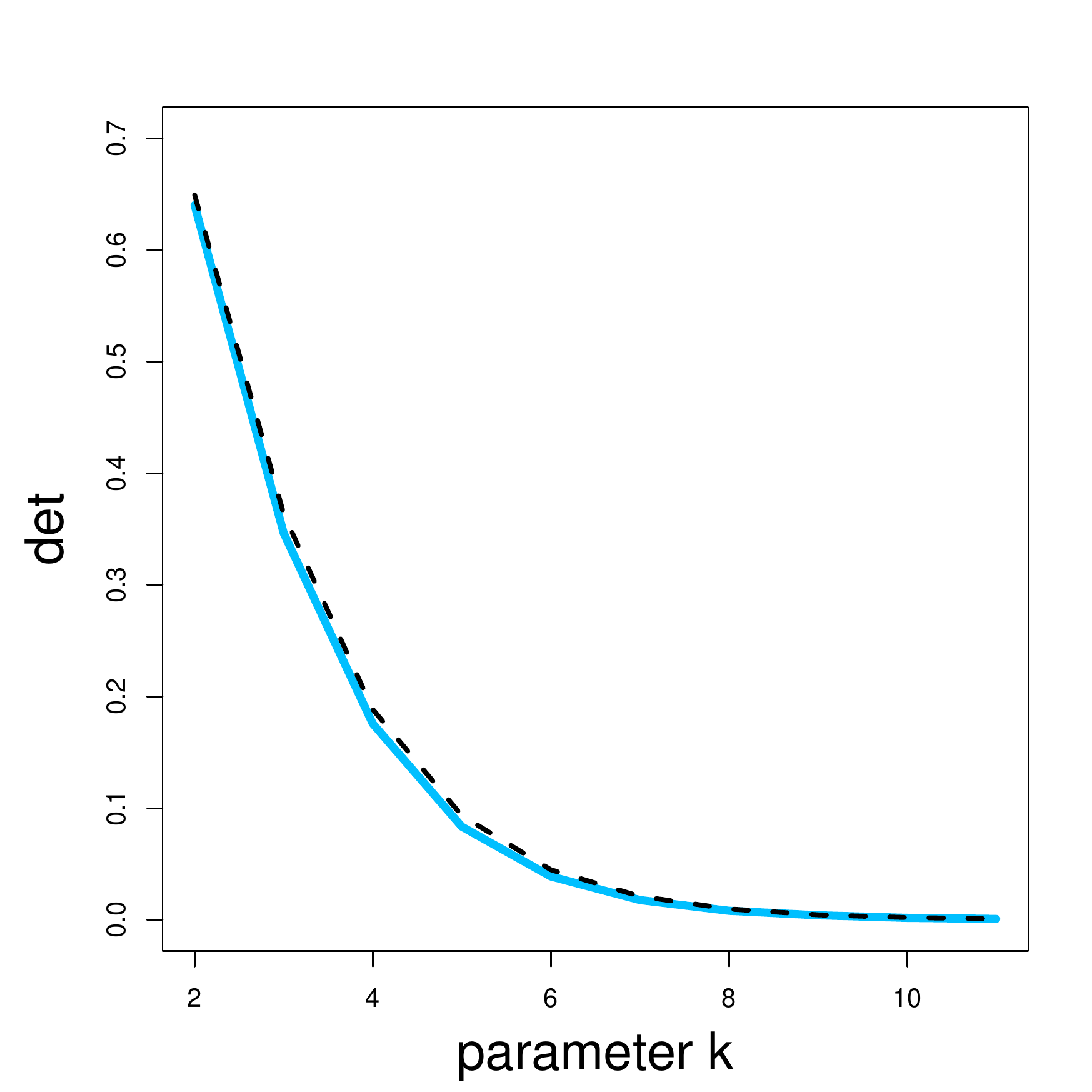}
  \includegraphics[width=4.5cm]{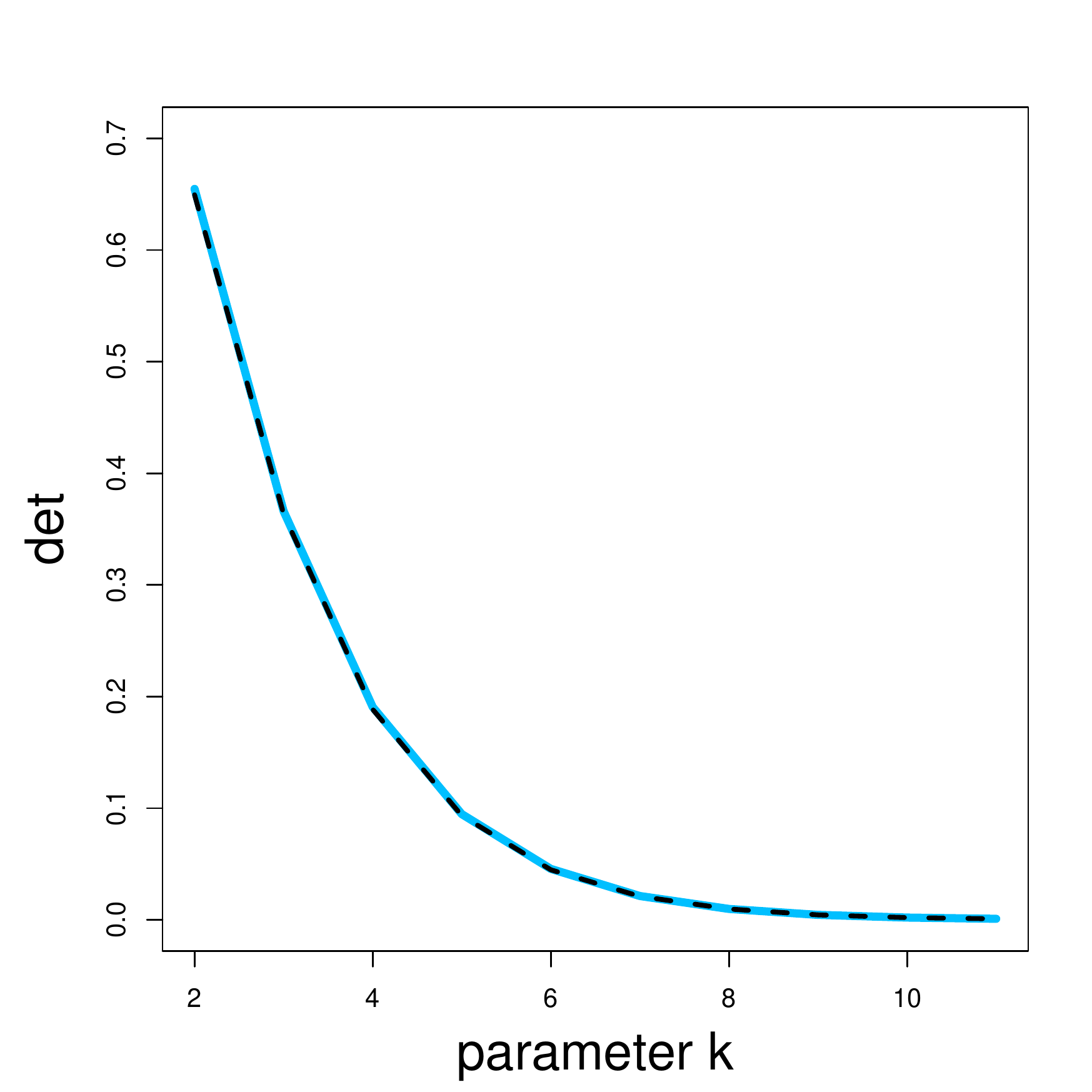}\\
  \caption{Convergence of the empirical determinism (solid line)
    to the asymptotic determinism (dashed line)
    for a $3$--state Markov chain with the transition matrix $P$;
    $n=100, 1\,000, 10\,000$.
  }
  \label{fig:convergence-mc}
\end{figure}

\begin{proof}
Only the equality for $\c^m_k(r)$ needs a proof since the other two follow from
(\ref{eq:RQA-asympt-def}).
We may assume that $m=1$. Let $Y_0^k,Z_0^k$ be independent random vectors with the distribution
equal to that of $X_0^k$. Fix $r<1$ and put $\alpha=\c(r)$; then
$\alpha=\Prmu\{d(Y_0,Z_0)\le r\} = \sum_s \Prmu\{Y_0=Z_0=s\}=\pi'\pi$.

If $k=1$ we are done. So assume that $k\ge 2$ and put
$\beta=\c_{1|1}(r)=\Prmu\{d(Y_1,Z_1)\le r ~|~ d(Y_0,Z_0)\le r\}$.
Then
\begin{eqnarray*}
 \beta
 &=&
 \frac 1\alpha  \,\Prmu\{Y_0=Z_0, Y_1=Z_1\}
 =
 \frac 1\alpha \sum_{s,t} \Prmu\{Y_0=Z_0=s, Y_1=Z_1=t\}
\\
 &=&
 \frac 1\alpha \sum_{s,t} \Prmu\{X_0=s, X_1=t\}^2
 =
 \frac 1\alpha \sum_{s,t} \pi_s^2 \cdot p_{st}^2
 =
 \frac 1\alpha \left(
   \pi'\textnormal{diag}(PP')\pi
 \right) .
\end{eqnarray*}

Since $X_0^\infty$ is a stationary Markov chain, we obtain
\begin{eqnarray*}
 \c_k(r)
 &=&
 \Prmu\{d_k(Y_0^k,Z_0^k)\le r\}
 =
 \Prmu\{Y_i=Z_i \ \forall 0\le i<k \}
\\
 &=&
 \Prmu\{Y_{k-1}=Z_{k-1} ~|~ Y_i=Z_i \  \forall 0\le i<k-1 \} \cdot \c_{k-1}(r)
 =\beta \c_{k-1}(r) .
\end{eqnarray*}
Now a simple induction gives the desired result.
\end{proof}

Figure~\ref{fig:convergence-mc} depicts
the convergence of the empirical determinism (with $m=1$)
to the asymptotic one for a $3$--state Markov chain with the
(randomly selected) transition matrix
$P = \left(
  \begin{array}{ccc}
    0.362 & 0.438 & 0.200 \\
    0.484 & 0.447 & 0.069 \\
    0.120  & 0.503 & 0.377\\
  \end{array}
\right)$.

\subsection{Autoregressive processes}\label{sec:ar}

Next we consider asymptotic RQA characteristics
of a (stationary) autoregressive process $X_0^\infty\sim AR(p)$ of order $p\ge 1$
with coefficients $\theta_i$ ($i=1,\dots,p$) and
with Gaussian zero mean noise $\eps_0^\infty$ of variance $\sigma^2$. It is given by
 $X_{n} = \theta_1\cdot X_{n - 1} + \theta_2\cdot X_{n - 2} +
 \dots + \theta_{p}\cdot X_{n-p} + \varepsilon_{n}$.

\begin{proposition}\label{prop:det-ar}
Let $X_0^\infty$ be an (ergodic stationary) autoregressive process
$AR(p)$ with coefficients $\theta_1,\dots,\theta_p$ and Gaussian $WN(0,\sigma^2)$.
Let $r>0$, $m,k \ge 1$ and $h=k+m-1$. Then
\begin{equation*}
\c_k^m(r) = \Prmu\{Y_0^{h}\in[-r, r]^{h}\},
\end{equation*}
where $Y_0^{h} \sim N(0, \Sigma)$ with $\Sigma$ being the $h\times h$ autocovariance
matrix of an $AR(p)$ process
with coefficients $\theta_1,\dots,\theta_p$ and Gaussian $WN(0,2\sigma^2)$.
\end{proposition}
\begin{proof}
Since the difference of two independent $AR$ processes with the same parameters
$\theta_1,\dots,\theta_p,\sigma^2$ is an $AR(p)$ process with
the same coefficients and noise variance $2\sigma^2$,
the statement immediately follows from
(\ref{eq:cmk-via-prob}).
\end{proof}

\begin{figure}[ht]
  \centering
  \includegraphics[width=4.5cm]{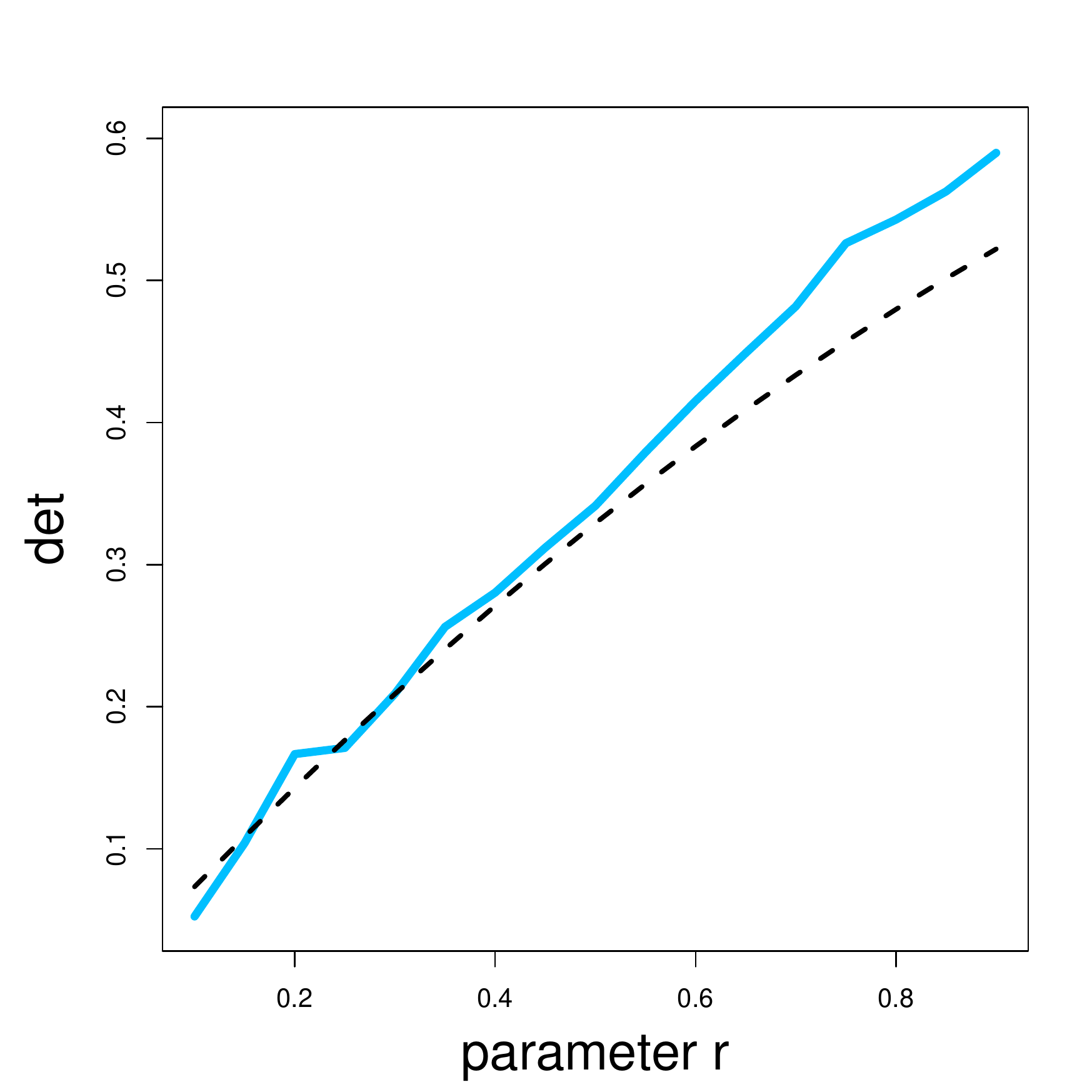}
  \includegraphics[width=4.5cm]{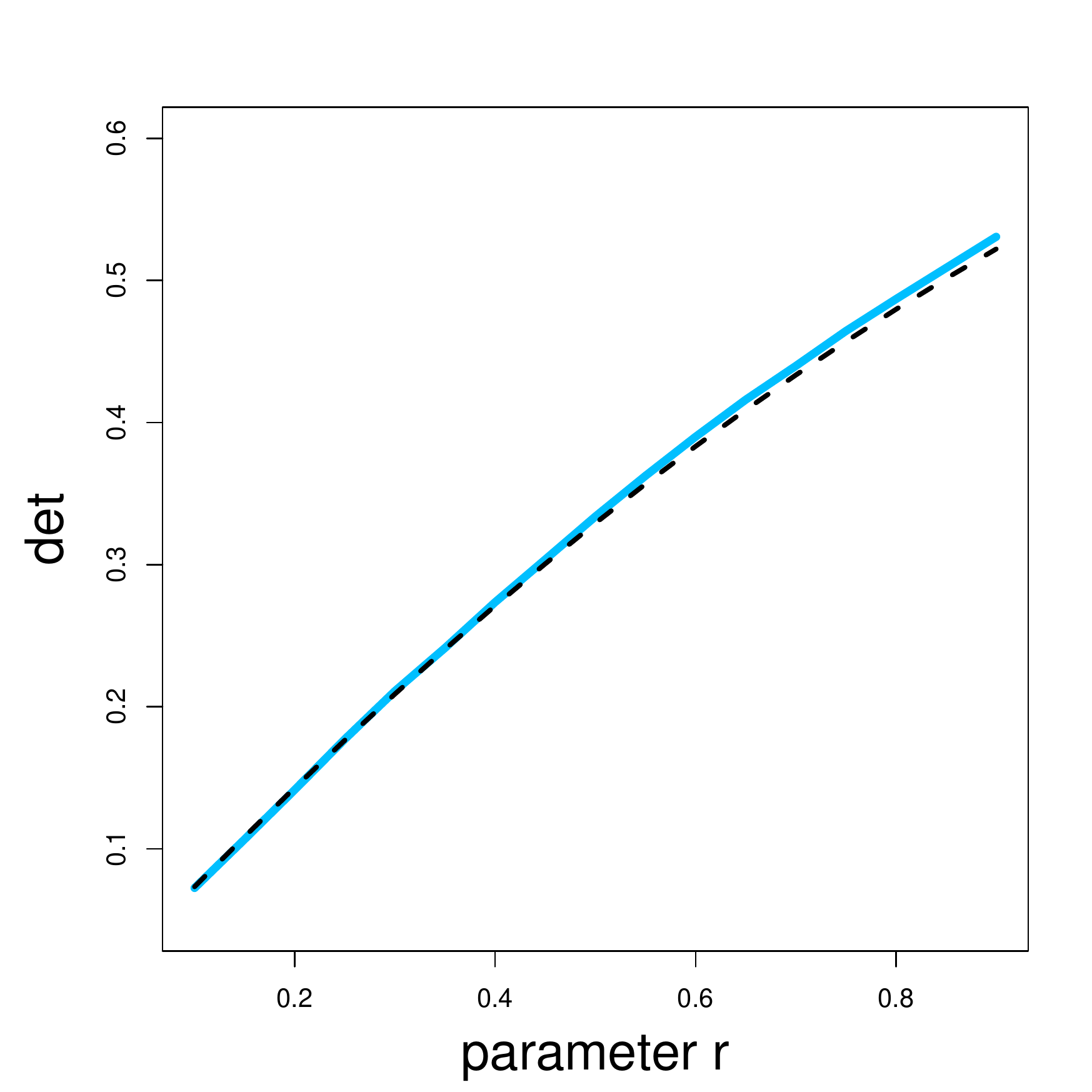}
  \includegraphics[width=4.5cm]{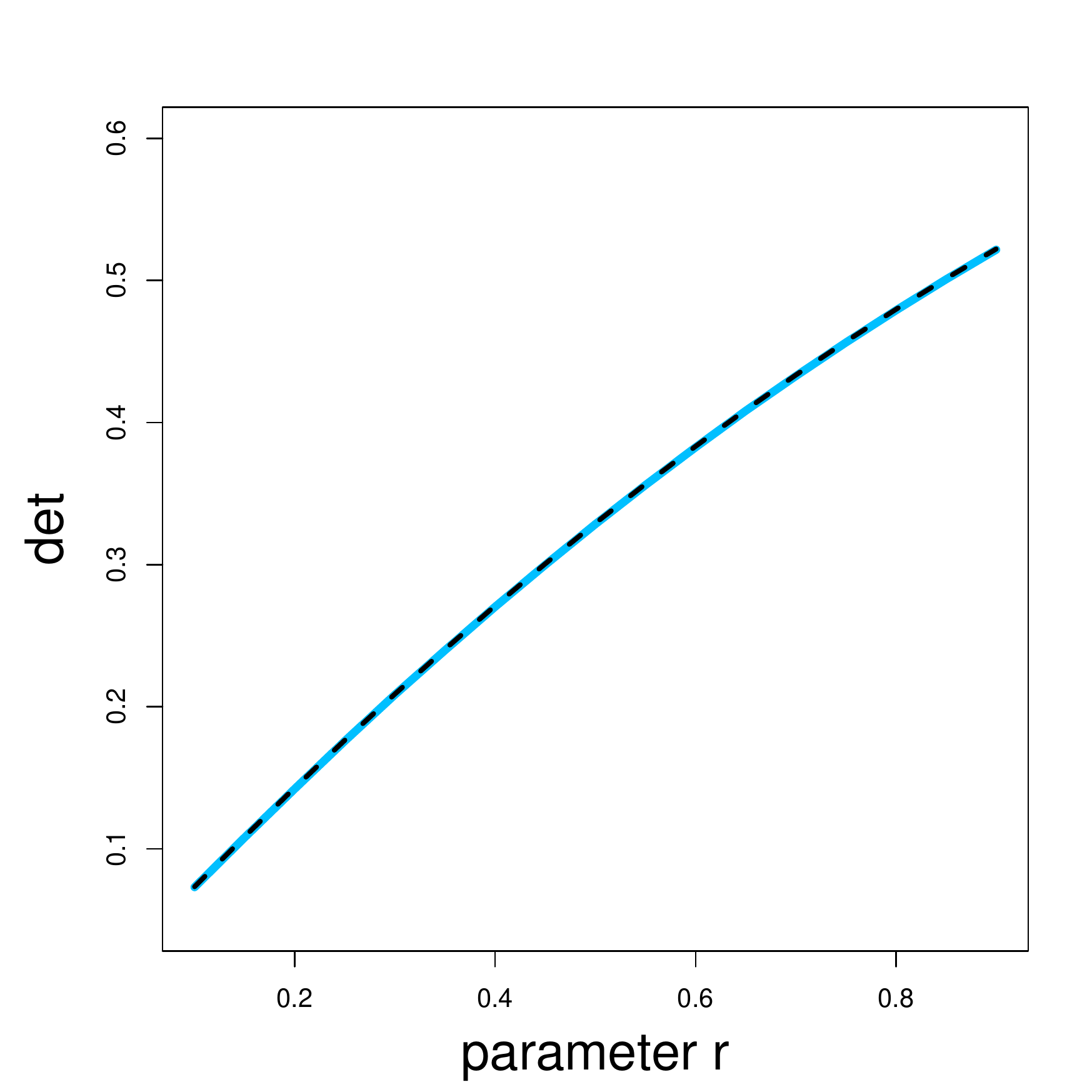}
  \caption{Convergence of the empirical determinism (solid line)
    to the asymptotic determinism (dashed line)
    for an $AR(3)$ process with parameters $a_1 = 0.25, a_2 = 0.4, a_3 = 0.3$ and $\sigma^2 = 1.5$;
    $n=100, 1\,000, 10\,000$.
  }
  \label{fig:convergence-ar}
\end{figure}

The convergence of the empirical determinism $\DET^1_2$
to the asymptotic one for an AR process is exhibited in
Figure~\ref{fig:convergence-ar}.

Corollary~\ref{cor:Markov-over-embedding} implies that over-embedding of an $AR(p)$
process into dimension $m>p$ leaves the determinism unchanged.
Thus, the asymptotic determinism can be used
to estimate (from below) the order of an autoregressive process.
This is demonstrated in Table~\ref{tabulka-ar} on an $AR(3)$ process. There one can see that
embedding into dimension $4$ or $5$ gives the determinism equal to that
for dimension $3$, but the determinisms for $m=1,2$ are smaller.
Thus one can conclude that the order of the process is at least $3$.

\newcommand{\art}[2]{$\stackrel[(#2)]{}{#1}$}
\begin{table}[ht!!!]
\centering
\begin{tabular}{l||ccccc}
\hline\hline
\backslashbox{k}{m}        &  $1$   &  $2$   &  $3$   &  $4$  &  $5$  \\ [0.5ex]
\hline\hline
 $2$   & \art{0.638}{0.009} & \art{0.725}{0.008} & \art{0.759}{0.008} & \art{0.760}{0.008} & \art{0.760}{0.009} \\
 $3$   & \art{0.407}{0.010} & \art{0.491}{0.011} & \art{0.514}{0.012} & \art{0.515}{0.013} & \art{0.516}{0.014} \\
 $4$   & \art{0.258}{0.010} & \art{0.312}{0.011} & \art{0.327}{0.012} & \art{0.328}{0.013} & \art{0.329}{0.014} \\
 $5$   & \art{0.158}{0.008} & \art{0.191}{0.010} & \art{0.201}{0.011} & \art{0.201}{0.012} & \art{0.202}{0.013} \\
[1ex]
\hline\hline
\end{tabular}
\caption{Determinisms for $AR(3)$ process with $a_1 = 0.25, a_2 = 0.4, a_3 = 0.3$ and
 white noise's variance $1.5$; here $n = 2\,500$ and $r = \sqrt{1.5}$.
 Average determinisms with standard errors in parentheses
 obtained by a Monte Carlo simulation of size $1\,000$.
}
\label{tabulka-ar}
\end{table}

\section{Kolmogorov entropy and asymptotic determinism}\label{sec:entropy-and-det}

In the following three examples we
demonstrate that behavior of the RQA determinism can sometimes be counterintuitive. First
we show that the determinism of an iid process can be higher than that of a non-iid one with the
same marginal. In the second example it is shown
that higher entropy does not necessarily mean smaller determinism.
Finally, a Markov chain indistinguishable (from the RQA point of view) from an iid process is constructed.

\begin{example}[Determinism of iid and non-iid processes]\label{ex:markov-iid}
Fix $0<a,b<1$ and consider a $01$-valued Markov chain $X_0^\infty$
with the transition matrix $P=(p_{st})_{s,t=0}^1$ such that $p_{00}=a, p_{11}=b$.
Then $X_0^\infty$ is ergodic and the stationary distribution of it is given by
 $\pi = \left(\frac{1 - b}{2 - a - b}, \frac{1 - a}{2 - a - b}\right)'$. Fix any $r\in [0,1)$.
 By Proposition~\ref{prop:det-markov},
\begin{equation*}
\det^m_k(r)^{\textnormal{Markov}}
=
k\beta^{k-1}-(k-1)\beta^k,
\qquad\text{where}\quad
  \beta=\c_{1|1}(r)=\tfrac{(1 - a)^2\cdot\left[b^2 + (1 - b)^2\right]
  + (1 - b)^2\cdot\left[a^2 + (1 - a)^2\right]}
  {(1 - a)^2 + (1 - b)^2}\,.
\end{equation*}
On the other hand, for a $01$-valued iid process with the same marginal $\pi$,
Proposition~\ref{prop:det-iid} gives
\begin{equation*}
\det^m_k(r)^{\textnormal{iid}}
=k\alpha^{k-1}-(k-1)\alpha^k,
\qquad\text{where}\quad
\alpha= \tfrac{(1 - a)^2 + (1 - b)^2}{(2 - a - b)^2}\,.
\end{equation*}
If we take $a=3/5$ and $b=1/5$, then $\alpha>\beta$. 
Since
the function $x\mapsto k x^{k - 1} - (k - 1) x^k$ is increasing on $[0, 1]$,
we have that
$\det^m_k(r)^{\textnormal{Markov}} > \det^m_k(r)^{\textnormal{iid}}$
for any $m,k\ge 1$.
\end{example}

\begin{figure}[ht]
  \centering
  \includegraphics[width=9cm]{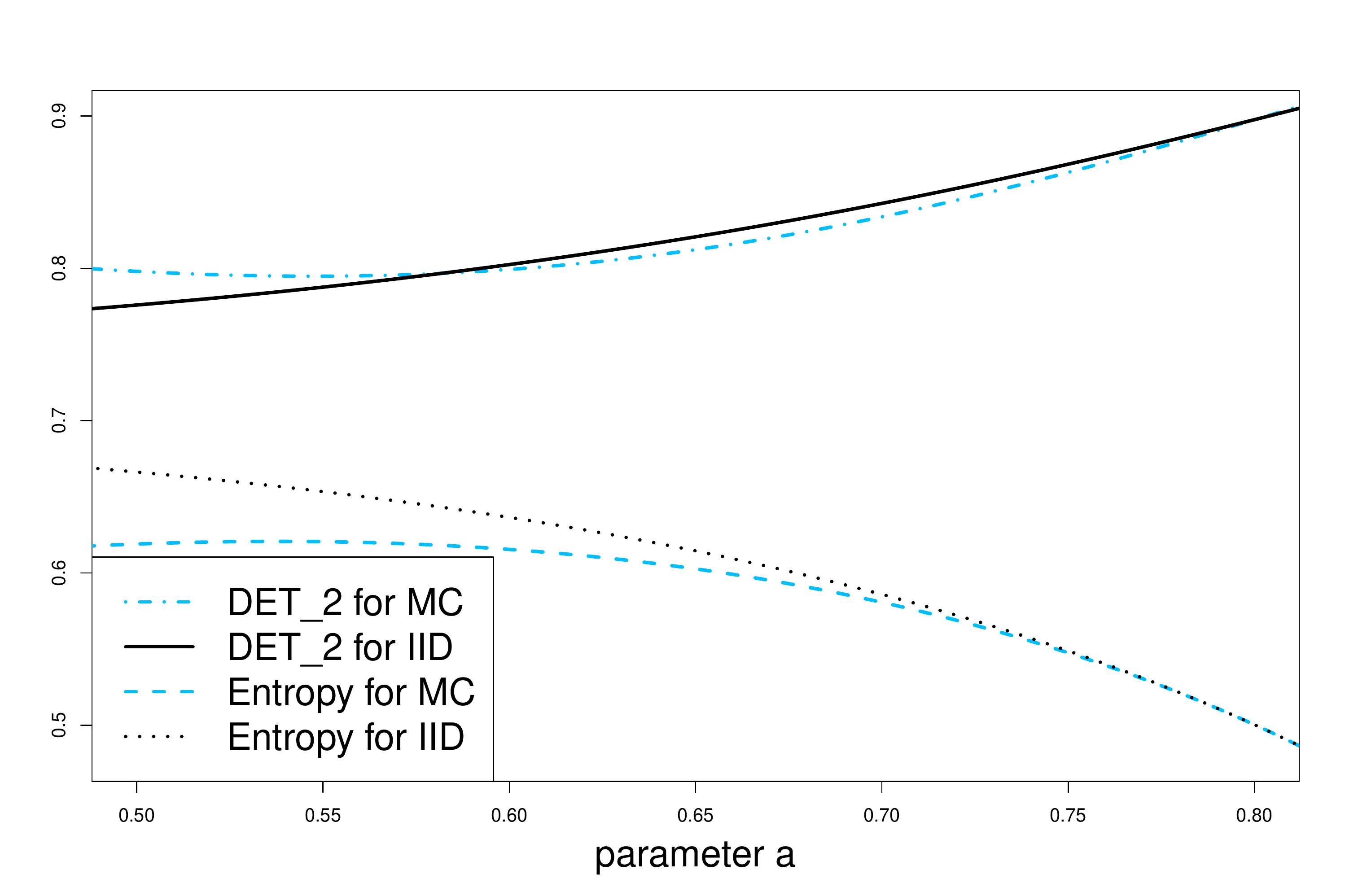}
  \caption{Entropy versus determinism  $\det^1_2$ for a $2$-state Markov chain (MC) with $b=1/5$
  and an iid process with the same marginal.}
  \label{fig:entropy-vs-determinism}
\end{figure}

\begin{example}[Determinism and entropy]
The previous example also shows that higher entropy does not necessarily mean smaller determinism.
In fact, the entropy of an iid process is strictly larger than that of any stationary non-iid process with the same
marginal.
In this simple case the entropies can be calculated
analytically, since the entropy of an iid process is
$h^{\textnormal{iid}}=-\sum_s \pi_s \log \pi_s$ and the entropy of the Markov chain
is $h^{\textnormal{Markov}}=-\sum_{st} \pi_s p_{st}\log p_{st}$.
See also Figure~\ref{fig:entropy-vs-determinism} for an illustration of this phenomenon.
\end{example}

\begin{example}[Indistinguishable Markov chain and iid process]
In the $2$-state Markov chain considered in Example~\ref{ex:markov-iid}, fix $b=1/5$
and, for given $a$, denote by $\alpha_a,\beta_a$ the corresponding correlation
integrals $\c_1(r),\c_{1|1}(r)$, respectively.
Since $\alpha_{1/2}<\beta_{1/2}$ and
$\alpha_{3/5}>\beta_{3/5}$, 
there is $a\in (1/2,3/5)$ with $\alpha_a=\beta_a$. 
For this particular (non-iid) Markov chain $X_0^\infty$, the probability of finding a diagonal line
of length $k$ (in the infinite recurrence plot) is the same as that for an iid process $Y_0^\infty$
with the same marginal. Hence, no RQA measure based
on diagonal lines can distinguish between $X_0^\infty,Y_0^\infty$.
\end{example}

\section{Spurious structures}\label{sec:spurious}

In \cite{Thiel2}, see also \cite[Section~3.2.4]{Marwan1}, it was pointed out that,
for iid processes, over-embedding leads to existence of spurious structures in recurrence plots.
The appearance of spurious structures is illustrated in Figure~\ref{fig:runif}.
The left panel depicts the ``usual'' recurrence plot
of an iid process for the embedding dimension $m=1$.
On the right panel there is the recurrence plot for the embedding dimension $m=250$.
It contains long diagonal lines, which would suggest that the process
should be well predictable.

\begin{figure}[ht]
  \centering
  \includegraphics[width=5cm]{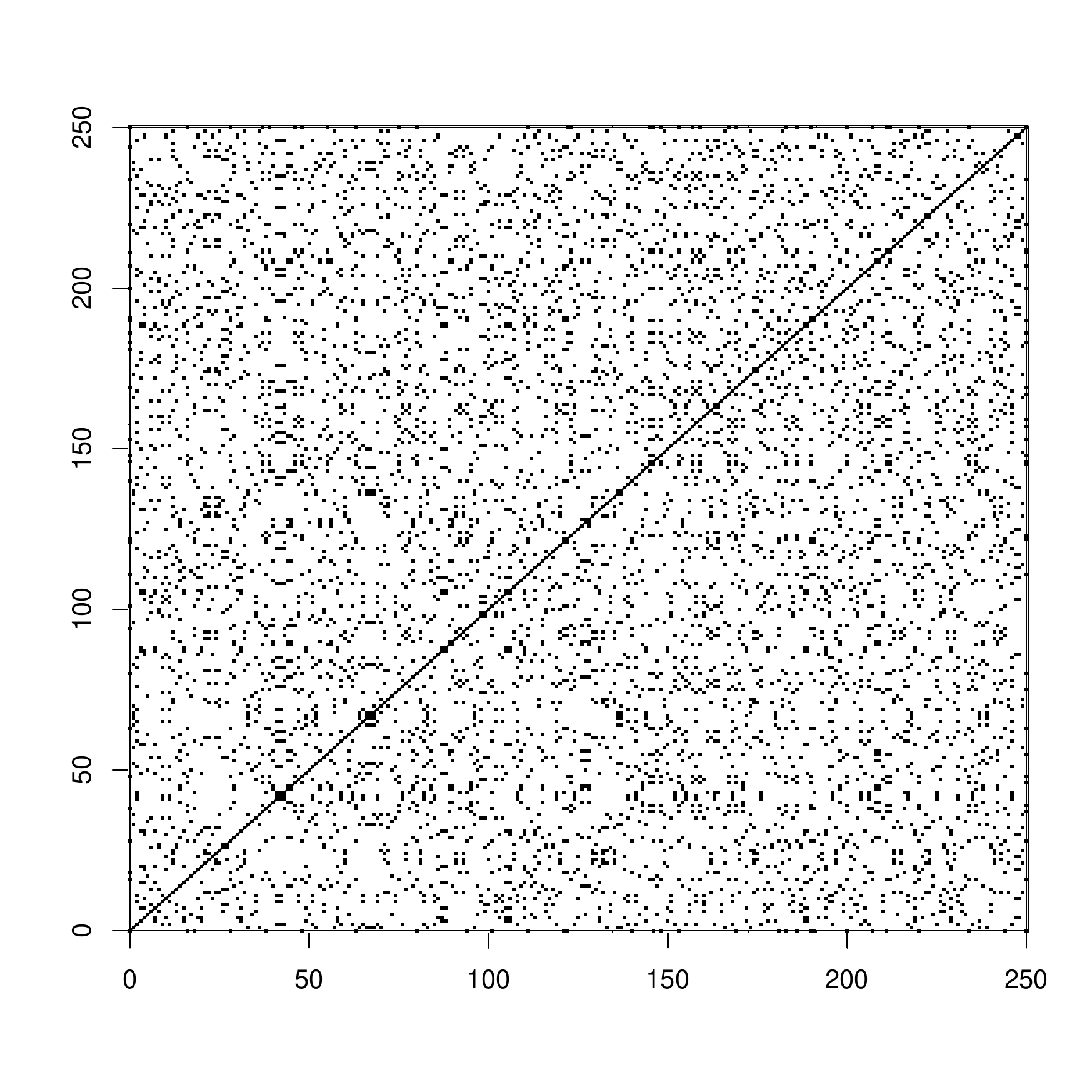}
  \includegraphics[width=5cm]{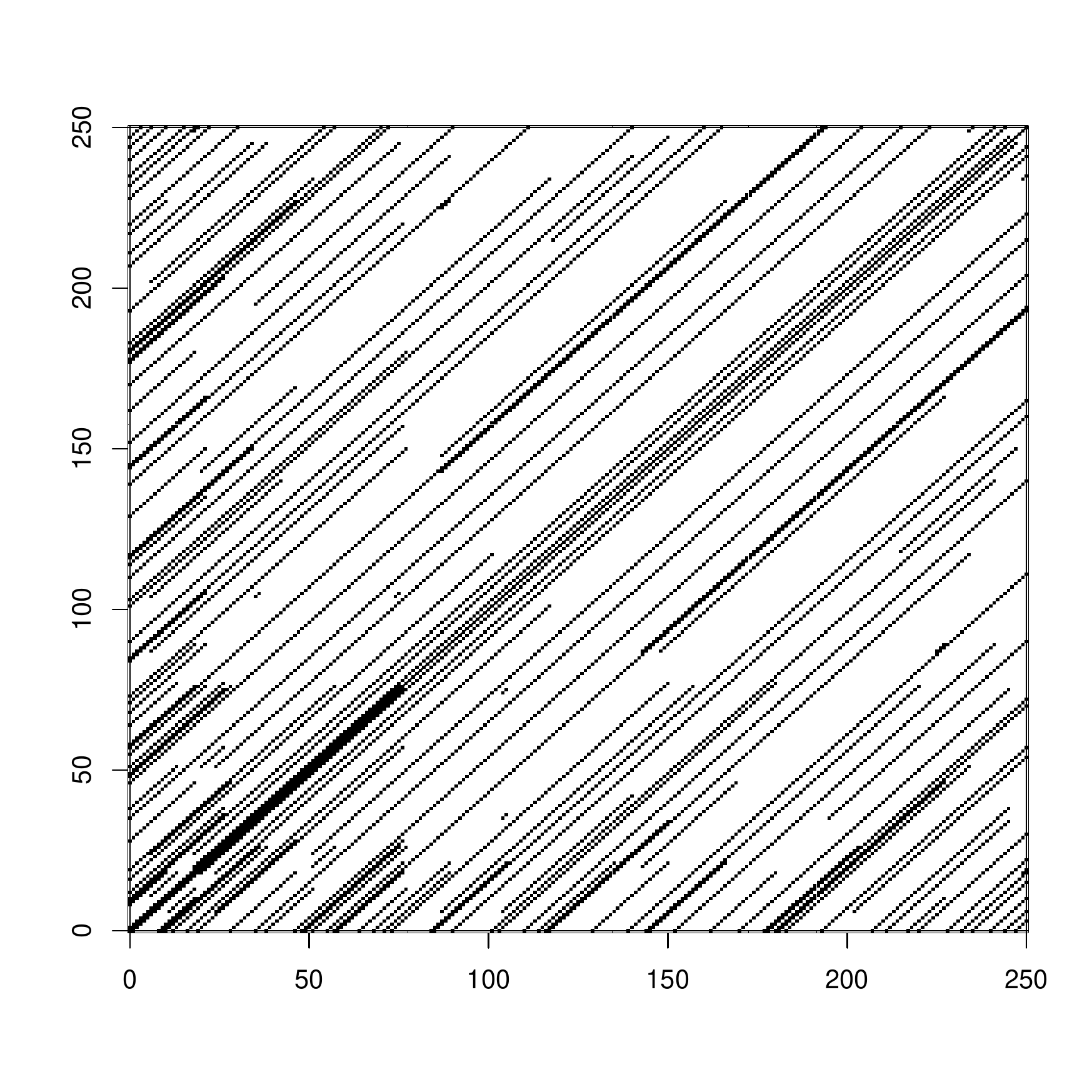}
  \caption{The effect of equal recurrence rates for embedding dimensions
    $m=1$ and $250$; uniform iid data.}
  \label{fig:runif}
\end{figure}

Proposition~\ref{prop:det-iid} enables us to explain why this happens. In fact, this is due to
a special choice
of the distance threshold $r$, which selects such $r=r_m$ that the recurrence rate
$\RR^m_1(r_m)$ is fixed to a predetermined level.
As the following proposition demonstrates, this selection rule leads to the determinism
close to one
and average diagonal line length arbitrarily high
for large embedding dimensions.

\begin{proposition}\label{prop:spurious}
Let $X_0^\infty$ be an iid process. Let $\theta>0$ and let $r_m>0$ ($m\in\NNN$)
be such that all the recurrence rates $\rr^m_1(r_m)$ are equal to $\theta$. Then, for $k\ge 1$,
$$
 \lim_{m\to\infty} \det^m_k(r_m)=1
 \quad\text{and}\quad
 \lim_{m\to\infty} \lavg^m_k(r_m)=\infty
 .
$$
\end{proposition}
\begin{proof}
For $m\ge 1$ put $\alpha_m=\c(r_m)$. Then, by the assumption and Proposition~\ref{prop:det-iid},
$\theta=\rr^m_1(r_m)=(\alpha_m)^m$ for every $m$; thus
$\alpha_m=\theta^{1/m}\to 1$ for $m\to\infty$. Using Proposition~\ref{prop:det-iid} we obtain that,
for every $k\ge 1$, $\rr^m_k(r_m)=(\alpha_m)^{m+k-1}\cdot [k-(k-1)\alpha_m] \to\theta$ for $m\to\infty$,
and so $\lim_{m} \det^m_k(r_m)=1$
and $\lim_{m} \lavg^m_k(r_m)=\infty$.
\end{proof}

Hence, appearance of the spurious structures for iid processes
is an artefact of this particular selection rule for density thresholds.
The artificial ``predictability'' which appears on the right panel of
Figure~\ref{fig:runif} is due to the distance threshold $r$,
which is several times higher than the standard deviation of
the process.
Different selection rule, which chooses $r$ independently of the embedding dimension,
leaves the determinisms $\det^m_k(r)$ and mean diagonal line lengths
$\lavg^m_k(r)$ constant for $m\to \infty$, as
one expects for iid processes.

\section{Strong law for correlation sums on pseudometric spaces}\label{sec:strong-law-c}

Here we give a proof of Theorem~\ref{thm:Cmk-strong-law}, based on the strong law
for correlation sums on pseudometric spaces; see Theorem~\ref{thm:C-strong-law} below.
Recall that, for a (topological) space $Z$, a map
$d:Z\times Z\to \RRR^+$ is a \emph{pseudometric} on $Z$ if
$d(x,x)=0$, $d(x,y)=d(y,x)$ and
$d(x,z)\le d(x,y)+d(y,z)$ for every $x,y,z\in Z$.
A pseudometric $d$ is \emph{separable} if the topology generated by it is separable.
We say that $d$ is a \emph{Borel (continuous) pseudometric} on $Z$ if it is a pseudometric
which is Borel (continuous) w.r.t.~the product topology on $Z\times Z$. Notice
that a continuous pseudometric on a separable space is automatically separable.

If $d$ is a pseudometric on $Z$, $B_d(x,r)$ and $S_d(x,r)$ denote the
(closed) $d$-ball and $d$-sphere with radius $r$ centered at $x$, respectively.
Notice that if $d$ is Borel then $d$-balls and $d$-spheres are
Borel sets in $Z$;
to see it, use that $B_d(x,r)=\{y:\ (x,y)\in d^{-1}([0,r])\}$ and analogously for $S_d(x,r)$.
The $d$-diameter of a set $A\subseteq Z$ is denoted by $\diam_d(A)$.

Assume that $(Z,\BBB_Z,\mu,T)$ is a dynamical system and that
$d$ is a Borel pseudometric on $Z$. For $x\in Z$, $n\in\NNN$ and $r\ge 0$
define the \emph{correlation sum}
\begin{equation*}
 \C_d(x,n,r)
 = \frac{1}{n^2} \card\{
  (i,j):\ 0\le i,j<n,\ d(T^i(x),T^j(x))\le r
 \}
\end{equation*}
and the \emph{correlation integral}
\begin{equation*}
 \c_d(r)
 = \mu\times\mu\{ (x,y):\ d(x,y)\le r\}
 = \int_Z \mu B_d(x,r) \, d\mu(x) .
\end{equation*}
Recall that $\c_d$ is non-decreasing, right continuous and tends to $1$ if $r\to\infty$.
Further, $\c_d$ is continuous at $r$ if and only if $\mu S_d(x,r)=0$ for $\mu$-a.e.~$x\in Z$,
see e.g.~\cite[Remark~2.2]{Pesin2}.

The strong law for the correlation sum was studied
under different conditions in \cite{Pesin1,Pesin2,Aaronson,Serinko,Manning}.
Though not stated in this form, the following theorem
was proved in \cite{Manning}.

\begin{theorem}\label{thm:C-strong-law}
Let $Z$ be a topological space, $\mu$ be a Borel probability  on $Z$
and $T:Z\to Z$ be a $\mu$-ergodic Borel map. Let $d$ be a separable Borel
pseudometric on $Z$.
Then, for $\mu$-a.e.~$x\in Z$ and for every $r>0$,
\begin{equation*}
 \lim_{n\to\infty} \C_d(x, n, r) = \c_d(r)>0
\end{equation*}
provided $\c_d$ is continuous at $r$.
\end{theorem}

Let us note that this ``pseudometric'' version of the strong law for
correlation sums cannot
be directly derived from the ``metric'' one.
Indeed, it is true that one can
easily obtain a metric space from the pseudometric one by gluing
together points of zero distance, as is usually done. The considered
dynamical system, however, does not necessarily fit to this projection
and so, in general, there is no induced system on the obtained metric space;
take e.g.~the case when $Z=\RRR^\infty$, $T$ is the shift and
$d(x_0^\infty,y_0^\infty)=\abs{x_0-y_0}$.

The proof from \cite{Manning}, however, perfectly fits to this general setting,
as was noted by the authors. Indeed, it is based on the Birkhoff ergodic theorem and on the
existence of finite Borel partitions $\AAA^m=\{A_j^m:\ 0\le j\le M_m\}$ ($m\ge 1$)
with $\mu(A_0^m)\le 2^{-m}$ and $\diam_d(A_j^m)\le 2^{-m}$ for every $j\ge 1$.
Since the former is true for arbitrary ergodic system and the latter immediately follows
from separability of $(Z,d)$ and Borel measurability of $d$-balls,
the convergence in Theorem~\ref{thm:C-strong-law} can be proved using
the same reasoning as in \cite{Manning}. Finally, the fact that
$\c_d(r)>0$ for every $r>0$ is obvious due to separability of $(Z,d)$.
(To see it, take any $d$-ball $B$ with radius $r/2$ and $\mu(B)>0$
and use that $\c_d(r)\ge \mu(B)^2$.)

\smallskip

Next we show how
Theorem~\ref{thm:Cmk-strong-law} can be obtained from Theorem~\ref{thm:C-strong-law}.

\begin{proof}[Proof of Theorem~\ref{thm:Cmk-strong-law}]
Let $(S,\varrho)$ be a separable metric space,
$m,k\ge 1$ be integers and $r>0$ be such that $\c^m_k$ is continuous at it.
Let $\varrho^m$ be a metric on $S^m$ compatible with the product topology.
Put $Z=S^\infty$ and define $d=d^m_k$ by (\ref{eq:dmk-def}).
Then obviously $d$ is a continuous pseudometric on $Z$; it is separable
due to separability of $Z$.

Let $X_0^\infty$ be an $S$-valued ergodic stationary process with distribution $\mu$.
We may assume that $X_0^\infty$ is given by its Kolmogorov representation, that is,
$X_n=\pi\circ T^n$, where $T:Z\to Z$ is the shift and $\pi:Z\to S$ is the projection
$x_0^\infty\mapsto x_0$. Then, for every $x=x_0^\infty\in Z$ and every $n$,
$T^n(x)=x_n^\infty$ and so
\begin{equation*}
 \c^m_k(r) = \c_d(r)
 \qquad\text{and}\qquad
 \C^m_k(x,n,r) = \left(\frac{n-k+1}{n}\right)^2 \C_d(x,n-k+1,r).
\end{equation*}
Application of Theorem~\ref{thm:C-strong-law} to the ergodic system
$(Z,\BBB_Z,\mu,T)$ gives the desired result.
\end{proof}

\subsection*{Acknowledgments.}
The authors gratefully acknowledge a substantive feedback from Lenka Mackovi\v cov\'a
and Jana \v Skutov\'a.
This paper was prepared as a part of the ``SPAMIA'' project,
M\v S~SR~3709/2010-11, supported by the Ministry of Education, Science, Research and Sport
of the Slovak Republic, under the heading of the state budget support for research and development.
Supported by the Slovak Grant Agency under the grant number VEGA~1/0978/11
and by the Slovak Research and Development Agency under the contract No.~APVV-0134-10.

\bibliography{refs}

\end{document}